\documentclass{article}
\usepackage{makeidx}
\usepackage{amssymb , amsmath, amsthm}
\usepackage{amsmath}
\usepackage{amssymb,latexsym}
\usepackage[dvips]{graphicx}
\usepackage{color}
\usepackage{variations}
\newcommand{\nero}{\smallskip$\bullet\quad$}
\newcommand{\h}{\mathbb H}
\newcommand{\R}{\mathbb R}

\newtheorem{theorem}{Theorem} [section]
\newtheorem{lemma}{Lemma} [section]
\newtheorem{proposition}{Proposition} [section]
\newtheorem{corollary}{Corollary} [section]

\newtheorem{remark}{Remark}[section]
\newcommand{\mnote}[1]
{\protect{\stepcounter{mnotecount}}$^{\mbox{\footnotesize
$
\bullet$\themnotecount}}$ \marginpar{
\raggedright\tiny\em $\!\!\!\!\!\!\,\bullet$\themnotecount: #1} }
\newcounter{mnotecount}[section]
\begin{document}

  \title{ Caccioppoli's inequalities on constant mean curvature hypersurfaces in Riemannian manifolds}

\author{S. Ilias, B. Nelli, M.Soret }
\date{ }
\maketitle

\begin{abstract} 
We prove some Caccioppoli's inequalities for the traceless part of the second fundamental form of a complete, noncompact, finite index, constant mean curvature hypersurface of a Riemannian manifold, satisfying some curvature conditions. This allows us to unify and clarify many results 
scattered 
in the literature and to obtain some new results. For example, we prove that there is no stable,  complete, noncompact hypersurface in ${\mathbb R}^{n+1},$ $n\leq 5,$ with constant mean curvature $H\not=0,$ provided  that, for suitable $p,$ the  $L^p$-norm of the traceless part of second fundamental form satisfies some growth condition.


\noindent
{\it M.S.C.  2010.} 53A10, 53C40, 53C42, 58E15.
\end{abstract}

\section{Introduction}
\label{introduction}

In this article,  we give a general setting that unifies  and clarifies  the proofs of many results present in the literature  on nonexistence of stable constant mean curvature hypersurfaces. We obtain some new results in the subject as well. The key result of this article is the following generalization of a result of R. Schoen, L. Simon, S.T. Yau \cite{SSY}.
Throughout the article,  ${\mathcal N}$ is  an  orientable Riemannian manifold with bounded sectional curvature  and such that the norm of the derivative of the curvature tensor is bounded. 

\ 

{\bf Theorem} (see Theorem \ref{general-ineq-theo}).
{\it Let $M$ be a complete, noncompact  hypersurface with constant mean curvature $H$ and finite index, immersed in  $\mathcal N.$ 
Denote  by $\varphi$ the norm of the traceless part of the second fundamental form of $M.$  
Then, there exists a compact subset $K$ of $M$, such that for any $q>0$ and  for any $f\in C_0^{\infty}(M\setminus K)$  one has 

\begin{align*}
\int_{M\setminus K} f^{2}\varphi^{2q+2}&({\mathcal A}\varphi^{2}+{\mathcal B}H\,\varphi+{\mathcal C}H^{2}+{\mathcal E})\notag \\ 
&\le {\mathcal D} \int_{M\setminus K}\varphi^{2q+2}|\nabla f|^{2}+{\mathcal F}\int_{M\setminus K} f^{2} \varphi^{2q+1}+ 
{\mathcal G} \int_{M\setminus K} f^{2} \varphi^{2q}
\end{align*}

where the constants are as in Theorem \ref{general-ineq-theo}. Moreover, if $M$ is stable, $K=\emptyset.$ 
}

\

The inequality of the previous Theorem is a Caccioppoli's inequality for finite index hypersurfaces with constant mean curvature.  In order to get it, we first prove  
a Simons' inequality for constant mean  curvature hypersurfaces of a Riemannian manifold (Theorem \ref{simon-gen-theo}).
Beyond their own  interest,  Caccioppoli's inequalities  are useful to deduce nonexistence results  for stable constant mean curvature hypersurfaces in space forms, under some restrictions on the dimension of the space  and on the growth of the $L^p$ norm of the traceless part of the second fundamental form, for suitable $p.$ For example,  we prove that there is no stable,  complete, noncompact hypersurface in ${\mathbb R}^{n+1},$ $n\leq 5,$ with constant mean curvature $H\not=0,$ provided  that, for suitable $p,$ the  $L^p$-norm of the traceless second fundamental form satisfies some growth condition (Corollary \ref{coro-docarmo-peng-H}). This is an answer to a do Carmo's question in a particular case (see pg. 133 in \cite{Do}).

It is worthwhile to point out that the analogous of the previous Theorem  can be proved in the setting of $\delta$-stable hypersurfaces, and for hypersurfaces with constant $H_r$-curvature, that is the $r$-th symmetric function of the principal curvatures (in \cite{ASZ} and \cite{ASZ1} one can find some related results).  Moreover, we restrict ourselves to hypersurfaces, but, in fact, the previous Theorem  can be easily adapted to submanifolds of any codimension with parallel mean curvature. 

We analyze  the   consequences of the previous Theorem in 
 the case where the ambient manifold has constant  curvature. In a forthcoming article we will analyze the case where  
the ambient space is 
either a product  or a warped product of constant  curvature manifolds (see \cite{BE} for related results in the minimal case). 

As the  proof of  the previous Theorem relies on a Simons' inequality in this setting, 
we recall some literature about the classical Simons' inequality. 
In the pioneer  paper  \cite{Si}, J. Simons proved an identity for the Laplacian of the norm of the second fundamental form of a minimal submanifold of the Euclidean space. Such identity is known as {\em Simons' formula.} Using  Simons' formula and some more work on minimal cones, J. Simons was able to deduce that, for $n\leq 7,$ the only entire solutions of the minimal surface equation in ${\mathbb R}^n$ are linear functions. Concerning the  restriction on the dimension in the result 
of Simons, we recall that E. Bombieri, E. De Giorgi and E. Giusti  \cite{BGG} proved that, for $n>7,$ there are entire solutions of the minimal surface equation that are not linear.  Then, Bernstein's question \cite{Bern} was completely answered. 

Later, there has been  a lot of work about Simons' formula. In \cite{SSY}, R. Schoen, L. Simon, S.T. Yau generalized Simons' formula to an inequality
(known as {\em Simons' inequality}) for minimal hypersurfaces in a  Riemannian manifold. Then, they applied  Simons' inequality in order to prove 
an estimate for the $L^p$ norm of the second fundamental form of a stable minimal hypersurface in a  Riemannian manifold, for a suitable 
$p$ (see Theorem 1 in \cite{SSY}). 
Among many interesting consequences of the $L^p$ estimate in \cite{SSY}, we point out the following one: there is no nontotally geodesic, area minimizing hypersurfaces of dimension $n\leq 5,$ in a flat Riemannian manifold.

Some authors proved generalizations of  Simons' inequality. We quote two important works in this direction : the article of
P. B\'erard,  \cite{B}, where the author deduced a very general  Simons' identity satisfied by the second fundamental form of a   submanifold of arbitrary codimension of a  Riemannian manifold and the article of K. Ecker and G. Huisken \cite{EH}, where a general Simons inequality is obtained for hypersurfaces of the  Euclidean space.

Our article is organized as follows.

 In Section \ref{stability}, we give some generalities about stability. In particular,  we   clarify the relation 
 between the different notions of index and stability. 

In Section \ref{simons-formula}, we obtain a Simons' inequality for the traceless part of the second  fundamental form of a constant mean curvature hypersurface in a Riemannian manifold 
(see Theorem \ref{simon-inequality}). We deduce it from the very general formula obtained by P. B\'erard in \cite{B1}, \cite{B}. 
Our computations are strongly inspired by  those of R. Schoen, L. Simon, S.T. Yau \cite{SSY}.  We give them, because, to our knowledge they are not present  in the literature, except for hypersurfaces in space forms \cite{AD1}, \cite{EH}.
 
 Then, we evaluate the terms of Simons' inequality depending on the curvature of the ambient space, in order to obtain a handier inequality (Theorem \ref{simon-gen-theo}). 
 
 In Section \ref{SSY-section}, we prove the Theorem that we stated at the beginning of the Introduction (Theorem \ref{general-ineq-theo}). 
 
 In Section \ref{Caccioppoli}, we deduce, from  Theorem \ref{general-ineq-theo}, three different kinds of Caccioppoli's inequalities on a constant mean curvature hypersurface with finite index.   The first one is the analogous, for finite index constant mean curvature hypersurfaces, of  the inequality obtained by R. Schoen, L. Simon, S. T. Yau (see Theorem 1 in \cite{SSY}).  The  second one is different in nature because it involves the gradient of the norm of the second fundamental form of the hypersurface. The third one is obtained by a careful study of the sign of the coefficients of the inequality of  Theorem  \ref{general-ineq-theo}, in the case where the ambient space has constant  curvature.  Later on,  we obtain some refinements of  Caccioppoli's inequality of third type, that will be useful for the applications.
 
 In Section \ref{applications}, we  show how our Cacioppoli's inequalities can be used to obtain some nonexistence results for stable constant mean curvature hypersurfaces. In many cases, we recover the known results present in the literature.

\section{ Stability notions and finite index hypersurfaces}
\label{stability}


The following assumptions will be maintained throughout this article.
  Let ${\mathcal N}$ be an  orientable Riemannian manifold and let  $M$  be an orientable  hypersurface  immersed in ${\mathcal N}.$ Assume that $M$ 
has  constant mean curvature. When the mean curvature is non zero,  we  orient $M$ by its mean curvature vector $\vec H.$ Then $\vec H=H\vec \nu$ and $H$ is positive.    When the mean curvature is zero, we choose, once for all, an orientation $\vec\nu$ on $M.$


  We introduce the {\em stability 
operator} $L,$ defined on  smooth functions  with compact support  in $M,$ that is  $L:=\Delta+Ric(\nu,\nu)+|A|^2,$ where $\Delta=tr\circ Hess$  and $A$ is the shape operator on $M.$

Let $\Omega$ be a relatively compact domain of $M.$ We denote by ${\rm Index}(\Omega)$ (respectively ${\rm WIndex}(\Omega)$) the number of negative eigenvalues of the operator  $-L,$ for the Dirichlet problem on $\Omega$ 

$$-Lf=\lambda f,\ \   f_{|\partial\Omega}=0$$
$$  ({\rm respectively}  \  -Lf=\lambda f,\   f_{|\partial\Omega}=0,\  \int_{\Omega}f=0).$$ 
The ${\rm Index}(M)$ (respectively ${\rm WIndex}(M)$) is defined as follows

$${\rm Index}(M):=\sup\{ {\rm Index}(\Omega) \ | \ \Omega\subset M \ {\rm rel.\ comp.}\}$$
$$({\rm respectively} \ {\rm WIndex}(M):=\sup\{ {\rm WIndex}(\Omega) \ | \ \Omega\subset M \ {\rm rel.\ comp.}\})$$

It is easy to see that $-\int_M fL(f)$ is the second derivative of the volume in the direction of $f\nu$ (see \cite{BDE}),   then ${\rm Index}(M)$  (respectively ${\rm WIndex}(M)$)  measures the  number of linearly  independent normal deformations with compact support of $M,$ decreasing  area
(respectively decreasing area, leaving fixed a volume). 
When $H=0,$ one can drop the condition $\int_{M}f=0,$ then, for a minimal  hypersurface, one considers only the ${\rm Index}(M).$  

The hypersurface $M$ is called {\em stable} (respectively {\em weakly stable}) if ${\rm Index}(M)=0$ (respectively  ${\rm WIndex}(M)=0$). This means that 

\begin{equation*}
Q(f,f):=-\int_MfL(f)\geq0,\ \ \forall f\in C_0^{\infty}(M)\ \ ({\rm respectively}\ \forall f\in C_0^{\infty}(M)\ \int_Mf=0).
\end{equation*}

It is proved in \cite{BB} that ${\rm Index}(M)$  is finite if and only if  ${\rm WIndex}(M)$  is finite. So,
when  we assume finite index, we
are referring to either of the indexes without distinction. 
Let us give some relations between stability and finite index.

\begin{proposition}
\label{stability-index-compact} Let $M$ be  a complete, noncompact constant mean curvature hypersurface in $\mathcal N.$ 

(1) If  $M$ is   weakly stable, then there exists a compact subset $K$ in $M$ such that  $M\setminus K$ is  stable.

(2) The hypersurface $M$ has finite index if and only if  there exists a compact subset $K$ in $M$ such that  
$M\setminus K$ is  stable.

\end{proposition}

\begin{proof} (1)   If $M$ is stable, we choose $K=\emptyset$ and (1) is proved. Assume $M$ is not stable, then there exists 
$f\in C^{\infty}_0(M)$ such that $Q(f,f)<0.$ Let $K=supp(f),$ we will prove that $M\setminus K$ is stable, i.e. 
for any $g\in C^{\infty}_0(M\setminus K),$ one has $Q(g,g)\geq 0.$ Denote by $\alpha=\int_M g$ and $\beta=\int_M f$ and define $h:=\alpha f-\beta g.$ By 
a straightforward computation, one has that $\int_M h=0.$ As $M$ is weakly stable, one has $Q(h,h)\geq 0.$ 
As $supp(f)\cap supp(g)=\emptyset,$ 
using  the bi-linearity of $Q$ one has
\begin{equation}
\label{bilinear}
0\leq Q(h,h)=\alpha^2 Q(f,f)+\beta^2 Q(g,g).
\end{equation}

As $Q(f,f)<0,$ inequality \eqref{bilinear} implies that $\beta\not=0$ and $Q(g,g)\geq 0.$ Hence $M\setminus K$ is stable.

(2) Assume that $M$ has finite index, then,  by a proof similar to the proof of Proposition 1 in \cite{FC}, one obtains that $M\setminus K$ is stable for a suitable compact subset $K.$ The vice versa is proved  by B. Devyver (see  Theorem 1.2 in \cite{De}).
\end{proof}

In the literature there are interesting relations between the stability  of a minimal hypersurface $M$  in $\mathbb R^{n+1}$ and the finiteness of $\int_M |A|^n.$ 
Let us recall some of them. 
P. Berard  \cite{B1} proved that  a complete stable minimal hypersurface of $\mathbb R^{n+1},$ $n\leq 5$ such that  $\int_M|A|^n<\infty,$ must be
a  hyperplane. 
In \cite{SZ}, Y.B. Shen and  X.H. Zhu stated that the previous result holds for any $n$  but in the proof, they use  an unpublished result by Anderson.
P. B\'erard, M. do Carmo and W. Santos  \cite{BDS} proved that,  if $M$ is a complete  hypersurface in 
$\h^{n+1},$ with constant mean curvature $H,$ $H^2<1,$ such that $\int_M|A-HI|^n<\infty,$ then $M$ has finite index. 
Notice that the converse is not true, as it is showed  by the examples by A. da Silveira \cite{S}.
 In \cite{Sp}, J. Spruck  proved that, if $M$ is a minimal submanifold of dimension $m$ of ${\mathbb R}^{n+1}$ such that   $\int_M |A|^m$  is small, then $M$  is stable 
(the definition of stability can be easily  extended to submanifolds of arbitrary codimension). In the same spirit of \cite{Sp}, one can prove the following 
well known result (we give a proof because we were not able to find one in the literature).

\begin{proposition}
 \label{bounded-finite}
Let $M$ be a complete minimal hypersurface  in ${\mathbb R}^{n+1},$  such that 
$\int_M |A|^n< \infty.$  Then  $M$ has finite index.

\end{proposition}
\begin{proof} 
As $M$ has infinite volume, the hypothesis yields that there exists a compact subset $K$ of $M$ depending on the Sobolev constant $c(n)$ such that 
inequality 
\begin{equation}
\label{small-curvature}
\int_{M\setminus K} |A|^n\leq {c(n)}^{-\frac{n}{2}}
\end{equation}

is satisfied. Then, one apply Theorem 2 in \cite{Sp}, to obtain that $M\setminus K$ is stable. Then, by (2) of Proposition 
\ref{stability-index-compact}, one obtains that 
$M$ has finite index.
\end{proof}

\begin{remark}
The proof of  Proposition \ref{bounded-finite}, can be  easily adapted to the case of an ambient manifold which  
is simply connected and of   curvature bounded from above by a nonpositive constant (see \cite{Fr} for the infinite volume argument).
\end{remark}

The idea of Spruck's proof of Theorem 2 in \cite{Sp} allows us to prove that if  $\int_MH^n$  small, without any further assumption on $M,$ one has a  lower bound on the volume of balls in $M.$

\begin{proposition}
\label{area}
Let $M$ be a complete submanifold  of dimension $n$ in a simply connected manifold  with   curvature bounded from above by a negative constant.  Denote by $H$ the mean curvature of $M$ and assume that 
$\left(\int_M|H|^n\right)^{\frac{1}{n}}\leq\frac{1}{2c},$ where $c$  is  a constant depending only on $n$ and on the bound on the curvature of the ambient space.
Then,  for any $R$ 

$$|B_R|\geq\frac{1}{(2cn)^n}R^n$$

where $|B_R|$ is  the volume of the  geodesic ball in $M$ of radius $R.$
\end{proposition}

\begin{proof}
Applying the isoperimetric inequality of D. Hoffman and J. Spruck \cite{HS} we obtain
\begin{equation}
\label{iso-sobolev}
|B_R|^{\frac{n-1}{n}}\leq c\left[|\partial B_R|+\int_{B_R}|H|\right].
\end{equation}

where $c$  is  a constant depending only on $n$ and on the bound on the curvature of the ambient space. 
By H\"older inequality one has

\begin{equation}\label{holder-sobolev}
\int_{B_R} |H|\leq  |B_R|^{\frac{n-1}{n}}\left(\int_{B_R} |H|^n\right)^{\frac{1}{n}}.
\end{equation}

Replacing \eqref{holder-sobolev} in \eqref{iso-sobolev} yields

\begin{equation}
\label{sobolev-2}
|B_R|^{\frac{n-1}{n}}(1-c\left(\int_{B_R} |H|^n\right)^{\frac{1}{n}})\leq c|\partial B_R|.
\end{equation}

By hypothesis, \eqref{sobolev-2} gives

 \begin{equation}\label{sobolev-3}
|B_R|^{\frac{n-1}{n}}\leq2c |\partial B_R|=2c\frac{d |B_R|}{d R}.
\end{equation}

By integrating inequality \eqref{sobolev-3} one has the result. 

\end{proof}

\begin{remark} It is clear that any minimal complete submanifold satisfies the assumptions of the Proposition \ref{area}.
\end{remark}



 \section{Simons'  inequality  for  constant mean curvature hypersurfaces}
 \label{simons-formula}
 
The following notations (that closely follow those of J. Simons \cite{Si} and  P. B\'erard \cite{B}) will be maintained throughout the article. 
Assume that $\mathcal N,$ $M,$ $H$  and $\nu$  are as at the beginning of Section \ref{stability}.
We denote by  $\overline \nabla$  the Levi-Civita connection in ${\mathcal N}$  and  by $R$ (respectively  $Ric$) the curvature tensor  (respectively  Ricci curvature)  of $\mathcal N.$ 
 We denote   
by $\nabla$ the connection on $M$ for the induced metric $g.$
   Let $A$ (respectively $B$) be the shape operator (respectively the  second fundamental form)  of $M,$ i.e. 
$\langle A(x), y\rangle=\langle \nu, B(x,y)\rangle$ for any  $x,y$ in $TM$. Furthermore, if $w$ is any  vector field normal to $M,$
 denote by $A^w$ the tensor field defined by $\langle A^w(x), y\rangle=\langle w, B(x,y)\rangle$ (note that with this notation, $A=A^{\nu}$).
Denote by $\Phi$ the traceless part of the second fundamental form, defined by $\Phi=B- H\nu g$ and   by $\phi$ the endomorphim of $TM$ 
associated $\Phi,$ i.e.   $\langle \phi(x), y\rangle=\langle \nu, \Phi(x,y)\rangle,$ for any $x,y$ in $TM.$
Finally, let $\nabla^{2}$ the rough Laplacian of the normal bundle, associated to the connection  $ \nabla,$ defined by 
$\nabla^{2}=\sum_{1}^{n} \nabla_{e_{i}}\nabla_{e_{i}}-\nabla_{\nabla_{e_{i}}e_{i}},$  
where $\{e_{i}\}$ is a local orthonormal frame in a neighborhood of a point of $M$.

From now on, the ambient manifold $\mathcal N$ will satisfy the following assumptions. 
We denote by $sec(X,Y)$ the sectional curvature of ${\mathcal N}$ for the two plane generated by $X,Y\in T{\mathcal N}.$ We assume that   $K_1$ and $K_2$ are an upper bound and a lower bound of the sectional curvatures,  i.e. for any $X,$ $Y$ $\in T{\mathcal N},$ $K_2\leq sec(X,Y)\leq K_1.$
Furthermore we assume that the derivative of the curvature tensor is bounded, that is,  there exists a constant $K'$ such that, for any elements  $e_{i},e_{j},e_{k},e_{s},e_{t}$ of a local orthonormal frame, one has  $\sqrt{\sum_{ijkst}\langle\left(\nabla_{e_t}R\right)(e_i,e_j)e_k,e_s}\rangle^2 \le K'.$

We will compute   $\langle\nabla^{2}\Phi,\Phi\rangle:=\sum_{1}^{n} \langle\nabla^2\Phi(e_i,e_j),\Phi(e_i,e_j)\rangle $  (where  $\{e_{i}\}$  is a local  orthonormal frame),  in terms of $|\phi|,\, H,\,  K_{1},\, K_{2} \;$  and $K'$ in order to obtain the following  Simon's type inequality.

\begin{theorem} [{\bf Simons' inequality}]
\label{simon-gen-theo} Let $M$ be a hypersurface of constant mean curvature  $H$ immersed in $\mathcal N$ and  let $\varphi=|\phi|.$ 
Then,  there exists $\varepsilon>0$ such that  the following inequality holds

\begin{align}
\label{simons5}
&\varphi\Delta\varphi\geq \frac{2}{n(1+\varepsilon)}|\nabla\varphi|^2-\varphi^4-\frac{n(n-2)H}{\sqrt{n(n-1)}}\varphi^3\notag
\\
&
+n\left(H^2+\frac{(K_2-K_1)H}{2}+2K_2-K_1\right)\varphi^2-2nK'\varphi
\\
&\notag
+\frac{n^2H(K_2-K_1)}{2}-\frac{n(n-1)}{2\varepsilon}(K_1-K_2)^2
\end{align}
\end{theorem}

 In order to prove Theorem \ref{simon-gen-theo}, we need some preliminary results. First, we have to establish a Simons' identity 
  for constant mean curvature  hypersurfaces.  
We prove it here for the sake of completeness but 
the proof is essentially contained in \cite{B}. 

   \begin{proposition} \label{FSimGen}
Let $M$ be a hypersurface of constant mean curvature, immersed in $\mathcal N.$ 
Then the following relation is satisfied at any point $p$ of $M$ 
\begin{align}
\label{total-Phi2}
&\notag 
\langle \nabla^2\Phi,\Phi\rangle=
-|\phi|^4+nHtr(\phi^3)+nH^2|\phi|^2-|\phi|^2 Ric(\nu,\nu)
+nH\sum_{i=1}^nR(\nu,e_i,\nu,\phi(e_i))
\\&
+\sum_{i,k=1}^n\{2R(e_k,\phi(e_i),e_k,\phi(e_i))+2R(e_k,e_i,\phi(e_i),\phi(e_k))\}
\\&\notag
+\sum_{i,k=1}^n
\{\langle(\overline\nabla_{e_i}R)(e_k,\phi(e_i))e_k,\nu\rangle+\langle(\overline\nabla_{e_i}R)(e_i,\phi(e_k))e_k,\nu\rangle\},
\end{align}

where   $\{e_{i}\}$  is a local  orthonormal frame at $p.$ 

Moreover, if $\mathcal N$ has constant  curvature $c,$ one has

\begin{equation}
\label{total-Phi2-cost}
 \langle\nabla^2\Phi,\Phi\rangle=
-|\phi|^4+nHtr(\phi^3)+n(H^2+c)|\phi|^2
\end{equation}

\end{proposition}

\begin{proof}

By definition

\begin{equation*}
\langle \nabla^2 \Phi, \Phi\rangle=\langle \nabla^2 B, \Phi\rangle=
\langle \nabla^2 B, B\rangle-H\langle\nabla^2 B, g\nu\rangle.
\end{equation*}

Then,  at $p$

\begin{equation}
\label{total-Phi}
\langle \nabla^2 \Phi, \Phi\rangle=\sum_{i,j=1}^n
\langle \nabla^2 B(e_i,e_j), B(e_i,e_j)\rangle-H\sum_{i=1}^n\langle\nabla^2 B(e_i,e_i), \nu\rangle.
\end{equation}
We need to compute the two terms in the right-hand side of  equation \eqref{total-Phi}. First we observe that, for any normal vector field $w$ and any tangent vector fields $x,$ $y$ one has (see Theorem 2 in \cite{B}) 

\begin{align}
\label{total-eq}
&\langle \nabla^2 B(x,y), w\rangle=-|A|^2\langle A^w(x),y\rangle +\langle  R(A)^w(x), y\rangle
\\
&
\notag
+\langle  R(nH\nu,x)y, w\rangle+\langle A^w(y),A^{nH\nu}(x)\rangle+\langle  R'^w(x), y\rangle,
\end{align}
where $R'^w$ and $R(A)^w$ are defined as follows:
\begin{equation}\notag
\langle  R'^w(x), y\rangle=\sum_{k=1}^n\{\langle(\nabla_xR)(e_k,y)e_k,w\rangle+\langle(\nabla_{e_k}R)(e_k,x)y,w\rangle
\}
\end{equation}
\begin{align}
&\langle R(A)^w(x), y\rangle=\sum_{k=1}^n\{2\langle R(e_k,y)B(x,e_k),w\rangle+2\langle R(e_k,x)B(y,e_k),w\rangle\notag
\\&\notag
-\langle A^w(x),R(e_k,y)e_k\rangle-\langle A^w(y),R(e_k,x)e_k\rangle
+\langle R(e_k, B(x,y))e_k,w\rangle-2\langle A^w(e_k),R(e_k,x)y\rangle
\}
\end{align}

  In order to  compute the first  term $\sum_{i,j=1}^n
\langle \nabla^2 B(e_i,e_j), B(e_i,e_j)\rangle$ in  \eqref{total-Phi}, we take $x=e_i,$ $y=e_j,$ $w=B(e_i,e_j)$ in \eqref{total-eq} and sum on $i$ and $j$. The computation of all the terms is as follows.

\

1. The first term is
\begin{equation} \label{first-B}
-|A|^2 \sum_{i,j}\langle A^{B(e_i,e_j)}(e_i),e_j\rangle=-|A|^2 \sum_{i,j}\langle A^2(e_j),e_j\rangle=-|A|^4.
\end{equation}

2.  The second  term is
\begin{align} \label{second-B}
 \sum_{i,j=1}^n\langle  R(A)^{B(e_i,e_j)}(e_i), e_j\rangle&\notag=
\sum_{k,i,j=1}^n\{2\langle R(e_k,e_j)B(e_i,e_k),B(e_i,e_j)\rangle+2\langle R(e_k,e_i)B(e_i,e_k),B(e_i,e_j)\rangle\\
&\notag-\langle A^{B(e_i,e_j)}(e_i),R(e_k,e_j)e_k\rangle-\langle A^{B(e_i,e_j)}(e_j),R(e_k,e_i)e_k\rangle\\
&+\langle R(e_k, B(e_i,e_j))e_k,B(e_i,e_j)\rangle-2\langle A^{B(e_i,e_j)}(e_k),R(e_k,e_i)e_j\rangle\}.
\end{align}

The first two terms in the  right-hand side of \eqref{second-B} are zero, then, rearranging terms

\begin{align} \label{second-B2}
\sum_{i,j=1}^n\langle  R(A)^{B(e_i,e_j)}(e_i), e_j\rangle&
=\sum_{k,i=1}^n\{2R(e_k,A(e_i),e_k,A(e_i))+2R(e_k,e_i,A(e_i),A(e_k))\}\notag
\\&-|A|^2Ric(\nu,\nu).
\end{align}

3. The third term is
\begin{align} \label{third-B}
\sum_{i,j=1}^n nH\langle  R(\nu,e_i)(e_j), B(e_i,e_j)\rangle=
nH\sum_{i=1}^n R(\nu,e_i,\nu ,A(e_i)).
\end{align}

4.  The fourth term is
\begin{align} \label{fourth-B}
\sum_{i,j=1}^n\langle A^{B(e_i,e_j)}(e_j),A^{nHN}(e_i)\rangle=nH\sum_{i=1}^n\langle A^3(e_i),e_i\rangle=nHtr(A^3).
\end{align}

5. The fifth term is
\begin{align} \label{fifth-B}
&\sum_{i,j=1}^n\langle  R'^{B(e_i,e_j)}(e_i), e_j\rangle=\sum_{i,k=1}^n
\{\langle(\overline\nabla_{e_i}R)(e_k,A(e_i))e_k,\nu\rangle+\langle(\overline\nabla_{e_i}R)(e_i,A(e_k))e_k,\nu\rangle\}.
\end{align}

\ 
 By summing up all the term in  \eqref{first-B}, \eqref{second-B2}, \eqref{third-B}, \eqref{fourth-B} and \eqref{fifth-B} and rearranging terms, one has

\begin{align}
\label{total-B}
&\notag \sum_{i,j=1}^n
\langle \nabla^2 B(e_i,e_j), B(e_i,e_j)\rangle
=-|A|^4+nHtr(A^3)+nH\sum_{i=1}^nR(\nu,e_i,\nu,A(e_i))
\\&
+\sum_{i,k=1}^n\{2R(e_k,A(e_i),e_k,A(e_i))+2R(e_k,e_i,A(e_i),A(e_k))\}-|A|^2Ric(\nu,\nu)
\\&\notag
+\sum_{i,k=1}^n
\{\langle(\overline\nabla_{e_i}R)(e_k,A(e_i))e_k,\nu\rangle+\langle(\overline\nabla_{e_i}R)(e_i,A(e_k))e_k,\nu\rangle\}.
\end{align}

In order  to compute the second  term $\langle\nabla^2 B(e_i,e_i), \nu\rangle$ of  \eqref{total-Phi},  we take  $x=y=e_i,$ and $w=\nu$  in 
\eqref{total-eq}. The computation of  all the terms  is as follows.

1. The first term is

\begin{equation}
\label{first-N}
-\sum_{i=1}^n|A|^2\langle A(e_i),e_i\rangle=-nH|A|^2.
\end{equation}

2.  The second term is

\begin{align}
\label{second-N}
&\sum_{i=1}^n\langle  R(A)^\nu(e_i), e_i\rangle=
\sum_{k,i=1}^n\{2\langle R(e_k,e_i)B(e_i,e_k),\nu\rangle+2\langle R(e_k,e_i)B(e_i,e_k),\nu)\rangle
\\&\notag
-\langle A(e_i),R(e_k,e_i)e_k\rangle-\langle A(e_i),R(e_k,e_i)e_k\rangle
+\langle R(e_k, B(e_i,e_i))e_k,\nu\rangle-2\langle A(e_k),R(e_k,e_i)e_i\rangle
\}.
\end{align}

The first two terms in the  right-hand side of \eqref{second-N} are zero, then

\begin{align}
\label{second-N1}
\sum_{i=1}^n\langle  R(A)^{\nu}(e_i), e_i\rangle
=\sum_{k,i=1}^n
\{-2\langle A(e_i),R(e_k,e_i)e_k\rangle-2\langle A(e_k),R(e_k,e_i)e_i\rangle\}-nHRic(\nu,\nu).
\end{align}

The  first two terms in the  right-hand side of \eqref{second-N1} are opposite, then

\begin{align}
\label{second-N2}
&\sum_{i=1}^n\langle  R(A)^{\nu}(e_i), e_i\rangle=-nHRic(\nu,\nu).
\end{align}

3. The third term is

\begin{align}
\label{third-N}
&\sum_{i=1}^n nH\langle  R(\nu,e_i)e_i, \nu\rangle=nH Ric(\nu,\nu).
\end{align}

4. The fourth term is

\begin{align}
\label{fourth-N}
\sum_{i=1}^n\langle A(e_i),A^{nH\nu}(e_i)\rangle=nH|A|^2.
\end{align}

5. The fifth term is

\begin{align}
\label{fifth-C}
\sum_{i=1}^n\langle  R'^{\nu}(e_i), e_i\rangle
=\sum_{i,k=1}^n\{\langle(\overline\nabla_{e_i}R)(e_k,e_i)e_k,\nu\rangle+\langle(\overline \nabla_{e_k}R)(e_k,e_i)e_i,\nu\rangle
\}.
\end{align}

\

 The sum of the terms in \eqref{first-N}, \eqref{second-N2}, \eqref{third-N} and \eqref{fourth-N} is zero, hence one has

\begin{align}
\label{total-N}
&\sum_{i=1}^n\langle\nabla^2 B(e_i,e_i), \nu\rangle=
\sum_{i,k=1}^n\{\langle(\overline\nabla_{e_i}R)(e_k,e_i)e_k,\nu\rangle+\langle(\overline \nabla_{e_k}R)(e_k,e_i)e_i,\nu\rangle
\}.
\end{align}

 Replacing  \eqref{total-B} and \eqref{total-N} in \eqref{total-Phi}, one obtains

\begin{align}
\label{total-Phi1}
&\notag\langle \nabla^2\Phi,\Phi\rangle=
-|A|^4+nHtr(A^3)+nH\sum_{i=1}^nR(\nu,e_i,\nu,A(e_i))
\\&\notag
+\sum_{i,k=1}^n\{2R(e_k,A(e_i),e_k,A(e_i))+2R(e_k,e_i,A(e_i),A(e_k))\}-|A|^2Ric(\nu,\nu)
\\&
+\sum_{i,k=1}^n
\{\langle(\overline\nabla_{e_i}R)(e_k,A(e_i))e_k,\nu\rangle+\langle(\overline\nabla_{e_i}R)(e_i,A(e_k))e_k,\nu\rangle\}
\\&\notag
-H\sum_{i,k=1}^n\{\langle(\overline\nabla_{e_i}R)(e_k,e_i)e_k,\nu\rangle+\langle(\overline \nabla_{e_k}R)(e_k,e_i)e_i,\nu\rangle
\}.
\end{align}
 In order to  obtain \eqref{total-Phi2}, one needs to  write the  right-hand side of  \eqref{total-Phi1} in terms of $\phi.$ This is straightforward 
 by replacing in \eqref{total-Phi1} the following identities

\begin{equation}
\notag
|A|^2=|\phi|^2+H^2n,\  |A|^4=|\phi|^4+n^2H^4+2nH^2|\phi|^2, \  tr(A^3)=tr(\phi^3)+nH^3+3H|\phi|^2.
\end{equation}

Finally, equality  \eqref{total-Phi2-cost} is a straightforward consequence of \eqref{total-Phi2}.
\end{proof}


A  key step towards Theorem \ref{simon-gen-theo} is the following Proposition.

\begin{proposition}
\label{simon-inequality}
Let $M$ be a hypersurface of constant mean curvature $H,$ immersed in $\mathcal N.$ 
Then the following relation is satisfied at   any point $p$ of $M$ and for any $\varepsilon>0$

\begin{align}
\label{simons3}
&|\phi|\Delta|\phi|+|\phi|^4+\frac{n(n-2)H}{\sqrt{n(n-1)}}|\phi|^3-nH^2|\phi|^2+|\phi|^2 Ric(\nu,\nu)
-nH\sum_{i=1}^nR(\nu,e_i,\nu,\phi(e_i))\notag \\
&-\sum_{i,k=1}^n\{2R(e_k,\phi(e_i),e_k,\phi(e_i))+2R(e_k,e_i,\phi(e_i),\phi(e_k))\}\\
&-\sum_{i,k=1}^n
\{\langle(\overline\nabla_{e_i}R)(e_k,\phi(e_i))e_k,\nu\rangle+\langle(\overline\nabla_{e_i}R)(e_i,\phi(e_k))e_k,\nu\rangle\} \notag \\
&\geq\frac{2}{n(1+\varepsilon)}|\nabla|\phi||^2-\frac{2}{\varepsilon} \sum_{ij}R(\nu,e_i,e_i,e_j)^2,\notag
\end{align}

where   $\{e_{i}\}$  is a local  orthonormal frame at $p.$
\end{proposition}

\begin{proof}

As in the Weitzenb\"ock formulas, by a straightforward computation, one has
\begin{align}\notag
\langle \nabla^2\Phi,\Phi\rangle=-|\nabla\Phi|^2+\frac{1}{2}\Delta|\Phi|^2
=-|\nabla\Phi|^2+|\Phi|\Delta|\Phi|+|\nabla|\Phi||^2
\end{align}

That is
\begin{align}
\label{simonsrough}
|\Phi|\Delta|\Phi|=\langle \nabla^2\Phi,\Phi\rangle+ |\nabla\Phi|^2 -|\nabla|\Phi||^2
\end{align}

In order to obtain  inequality  \eqref{simons3}, we will use a Kato's inequality  to estimate the difference $|\nabla\Phi|^2 -|\nabla|\Phi||^2.$   
R. Schoen, L.  Simon and  S. T. Yau \cite{SSY} did such estimate in the case of minimal hypersurfaces.  One can easily 
adapt their computation to the case $H\not=0,$ in order  to obtain the following result.

\begin{lemma}[{\bf Kato's inequality}]
\label{kato-lemma} Assume the hypothesis of Proposition \ref{simon-inequality} are satisfied. Then, for any positive $\varepsilon$ 
\begin{align}
\label{kato7}
|\nabla\Phi|^2 -|\nabla|\Phi||^2 \geq\frac{2}{n(1+\varepsilon)}|\nabla|\Phi||^2-\frac{2}{\varepsilon} \sum_{ij}R(\nu,e_i,e_i,e_j)^2.
\end{align}

Moreover, if   $\mathcal N$ has constant curvature, then $|\Phi|$ satisfies the simpler  inequality

\begin{align}
\label{kato8}
|\nabla\Phi|^2 -|\nabla|\Phi||^2 \geq\frac{2}{n}|\nabla|\Phi||^2.
\end{align}

\end{lemma}

Let us finish the proof of Proposition \ref{simon-inequality}.

Replacing \eqref{kato7} in \eqref{simonsrough} one has, for any positive $\varepsilon$

\begin{equation}
\label{simons1}
|\Phi|\Delta|\Phi|\geq \langle \nabla^2\Phi,\Phi\rangle
+ \frac{2}{n(1+\varepsilon)}|\nabla|\Phi||^2-\frac{2}{\varepsilon} \sum_{ij}R(\nu,e_i,e_i,e_j)^2.
\end{equation}
Replacing \eqref{total-Phi2} in \eqref{simons1} and using $|\Phi|=|\phi|$, we get

\begin{align} \label{simons2}
&|\phi|\Delta|\phi|
+|\phi|^4-nHtr(\phi^3)-nH^2|\phi|^2+|\phi|^2 Ric(\nu,\nu)
-nH\sum_{i=1}^nR(\nu,e_i,\nu,\phi(e_i))\notag
\\&
-\sum_{i,k=1}^n\{2R(e_k,\phi(e_i),e_k,\phi(e_i))+2R(e_k,e_i,\phi(e_i),\phi(e_k))\}
\\&\notag
-\sum_{i,k=1}^n
\{\langle(\overline\nabla_{e_i}R)(e_k,\phi(e_i))e_k,\nu\rangle+\langle(\overline\nabla_{e_i}R)(e_i,\phi(e_k))e_k,\nu\rangle\}
\\&\notag
\geq\frac{2}{n(1+\varepsilon)}|\nabla|\phi||^2-\frac{2}{\varepsilon} \sum_{ij}R(\nu,e_i,e_i,e_j)^2.
\end{align}

 Now, in order to estimate ${\rm tr}(\phi^{3})$ we need the following Lemma  by H. Okumura \cite{O}, \cite{AD}.

\begin{lemma}
The following algebraic inequality holds for any traceless operator $\phi:$ 
\begin{equation}
\notag
-\frac{n-2}{\sqrt{n(n-1)}}|\phi|^3\leq tr(\phi^3)\leq \frac{n-2}{\sqrt{n(n-1)}}|\phi|^3.
\end{equation}
\end{lemma}

Then, we replace the  inequality  of the last lemma in \eqref{simons2} and  obtain
\eqref{simons3}.

\end{proof}

When the ambient space has a particular geometry, one can simplify   inequality \eqref{simons3}.  
The first part of  next Corollary  is proved in \cite{AD1} (inequality (10) there).

\begin{corollary}
\label{coro-simons-spaceform-locsym}
(1) Assume that the ambient manifold has constant curvature $c.$ Then
\begin{align}
\label{simons-spaceform}
|\phi|\Delta|\phi|
+|\phi|^4
+\frac{n(n-2)H}{\sqrt{n(n-1)}}|\phi|^3
-n(H^2+c)|\phi|^2
\geq\frac{2}{n}|\nabla|\phi||^2.
\end{align}

(2) Assume that the ambient manifold  is locally symmetric, that is $\overline\nabla R\equiv 0.$ Then at any point $p\in M$ and any $\varepsilon>0$ 
\begin{align}
\label{simons-locsym}
&|\phi|\Delta|\phi|
+|\phi|^4
+\frac{n(n-2)H}{\sqrt{n(n-1)}}|\phi|^3
-n(H^2+Ric(\nu,\nu))|\phi|^2
-nH\sum_{i=1}^nR(\nu,e_i,\nu,\phi(e_i))\notag
\\&
-\sum_{i,k=1}^n\{2R(e_k,\phi(e_i),e_k,\phi(e_i))+2R(e_k,e_i,\phi(e_i),\phi(e_k))\}
\\&\notag
\geq\frac{2}{n(1+\varepsilon)}|\nabla|\phi||^2-\frac{2}{\varepsilon} \sum_{ij}R(\nu,e_i,e_i,e_j)^2.
\end{align}

where $\{e_i\}$ is a local orthonormal frame  at $p.$ 
\end{corollary}

Now we are ready to prove Theorem \ref{simon-gen-theo} (Simons' inequality). 

\begin{proof} [Proof of  Theorem  \ref{simon-gen-theo}]

We choose the local orthonormal frame such that it diagonalizes the endomorphism 
$\phi$ and we denote by $\lambda_{i}$ its eigenvalue associated to $e_{i}$ (i.e. $\phi(e_{i})=\lambda_{i}\,e_{i}$).
The proof is an estimation of the terms of \eqref{simons3} depending on $R$ and $\overline \nabla R.$ 

 The first term  of \eqref{simons3} to estimate is

\begin{align}
\label{curv-est1}
&\sum_{i=1}^nR(\nu,e_i,\nu,\phi(e_i))\notag\\
&\notag=\sum_{i=1}^nR
\left(\nu,\frac{e_i+\phi(e_i)}{\sqrt{2}},\nu,\frac{e_i+\phi(e_i)}{\sqrt{2}}\right)
-\frac{1}{2}\sum_{i=1}^n \left(R(\nu,e_i,\nu,e_i)+R(\nu,\phi(e_i),\nu,\phi(e_i))\right)
\\
&\notag
=\sum_{i=1}^n sec\left(\nu,\frac{e_i+\phi(e_i)}{\sqrt{2}}\right)\left\|\frac{e_i+\phi(e_i)}{\sqrt{2}}\right\|^2
-\frac{1}{2} Ric(\nu,\nu))-
\frac{1}{2} \sum_{i=1}^n sec(\nu,\phi(e_i))\|\phi(e_i)\|^2
\\
&\notag
=\sum_{i=1}^n sec\left(\nu,\frac{e_i+\phi(e_i)}{\sqrt{2}}\right)\frac{(1+\lambda_i)^2}{2}
-\frac{1}{2} Ric(\nu,\nu)-
\frac{1}{2} \sum_{i=1}^n sec(\nu,\phi(e_i))\lambda_i^2
\\
&
\geq \frac{nK_2}{2}-\frac{(K_1-K_2)}{2}|\phi|^2-\frac{1}{2} Ric(\nu,\nu).
\end{align}

The second term  of \eqref{simons3} to estimate is
\begin{align}
\label{curv-est2}
&\sum_{i,k=1}^n\{2R(e_k,\phi(e_i),e_k,\phi(e_i))+2R(e_k,e_i,\phi(e_i),\phi(e_k))\}
\notag\\
&
=\sum_{i,k=1}^n\{2\lambda_i^2R(e_k,e_i,e_k,e_i)+2\lambda_i\lambda_kR(e_k,e_i,e_i,e_k)\}
\\
&\notag
=\sum_{i,k=1}^n(\lambda_k-\lambda_i)^2sec(e_i,e_k)\geq K_2\sum_{i,k=1}^n(\lambda_k-\lambda_i)^2=2nK_2|\phi|^2.
\end{align}

 The third term of \eqref{simons3} to estimate is
\begin{align}
\label{curv-est3}
&-\sum_{i,k=1}^n
\{\langle(\overline\nabla_{e_i}R)(e_k,\phi(e_i))e_k,\nu\rangle+\langle(\overline\nabla_{e_i}R)(e_i,\phi(e_k))e_k,\nu\rangle\}
\notag\\
&
=-\sum_{i,k=1}^n
\{\lambda_i\langle(\overline\nabla_{e_i}R)(e_k,e_i)e_k,\nu\rangle+\lambda_k\langle(\overline\nabla_{e_i}R)(e_i,e_k)e_k,\nu\rangle\}
\\
&\notag
\leq 2\sum_k\sqrt{\sum_i\lambda_i^2}\sqrt{\sum_i \langle(\overline\nabla_{e_i}R)(e_k,e_i)e_k,\nu\rangle^2}
\leq 2nK'|\phi|.
\end{align}

where in the first inequality we  have used Cauchy-Schwarz inequality.

 The  fourth and last  term of  \eqref{simons3} to estimate  is $\sum_{ij}R(\nu,e_i,e_i,e_j)^2.$ One has

\begin{align}
\notag
R(\nu,e_i,e_j,e_i)=\frac{1}{2}\left(R\left(\frac{\nu+e_j}{\sqrt 2},e_i,\frac{\nu+e_j}{\sqrt 2},e_i\right)-
R\left(\frac{\nu-e_j}{\sqrt 2},e_i,\frac{\nu-e_j}{\sqrt 2},e_i\right)\right).
\end{align}

Hence

\begin{align}
\notag
\frac{K_2-K_1}{2}\leq R(\nu,e_i,e_j,e_i)\leq \frac{K_1-K_2}{2}.
\end{align}

That is

\begin{align}
\label{curv-est6}
\sum_{i,j}R(\nu,e_i,e_i,e_j)^2\leq\frac{n(n-1)}{4}(K_1-K_2)^2.
\end{align}

Replacing  \eqref{curv-est1}, \eqref{curv-est2}, \eqref{curv-est3}, \eqref{curv-est6} in  inequality \eqref{simons3}, one has
(recall that $\varphi=|\phi|$)

\begin{align}
\label{simons4}
&\varphi\Delta\varphi\geq -\varphi^4-\frac{n(n-2)H}{\sqrt{n(n-1)}}\varphi^3+nH^2\varphi^2
+nH\left(\frac{nK_2}{2}-\frac{(K_1-K_2)}{2}\varphi^2-\frac{1}{2} K_1n\right)
\\
&\notag
+n(2K_2-K_1)\varphi^2
-2nK'\varphi-\frac{n(n-1)}{2\varepsilon}(K_1-K_2)^2+\frac{2}{n(1+\varepsilon)}|\nabla\varphi|^2.
\end{align}

Rearranging terms in \eqref{simons4}, one obtains  \eqref{simons5}.

\end{proof}

Now, we state the result of Theorem 
\ref{simon-gen-theo}  in the particular case of $\mathcal N$ being a product of manifolds with constant  curvature. 
Notice that in this case $K'=0.$ 

\begin{corollary}
\label{product}
 Let $M_i(c_i)$ be  a Riemannian manifold with constant  curvature equal to $c_i=-1,0,1,$ $i=1,2.$ 
Let $M$ be a $n$-dimensional hypersurface immersed in a manifold $M_1(c_1)\times M_2(c_2)$ with constant mean curvature $H.$  
Then, for any $\varepsilon>0,$ we have
\

(1) $c_1=c_2=-1$ or $c_1=-1,$  $c_2=0:$
\begin{align}
\notag
\varphi\Delta\varphi\geq -\varphi^4-\frac{n(n-2)H}{\sqrt{n(n-1)}}\varphi^3+n(H^2-\frac{H}{2}-2)\varphi^2
-\frac{n}{2}\left(\frac{(n-1)}{\varepsilon}+nH\right)+\frac{2}{n(1+\varepsilon)}|\nabla\varphi|^2
\end{align}

(2) $c_1=-1,$  $c_2=1:$

\begin{align}
\notag
\varphi\Delta\varphi\geq -\varphi^4-\frac{n(n-2)H}{\sqrt{n(n-1)}}\varphi^3+n(H^2-H-3)\varphi^2
-n\left(\frac{2(n-1)}{\varepsilon}+nH\right)+\frac{2}{n(1+\varepsilon)}|\nabla\varphi|^2.
\end{align}

(3) $c_1=1,$  $c_2=1$ or  $c_1=1,$  $c_2=0:$

\begin{align}
\notag
\varphi\Delta\varphi\geq -\varphi^4-\frac{n(n-2)H}{\sqrt{n(n-1)}}\varphi^3+n(H^2-\frac{H}{2}-1)\varphi^2
-\frac{n}{2}\left(\frac{(n-1)}{\varepsilon}+nH\right)+\frac{2}{n(1+\varepsilon)}|\nabla\varphi|^2.
\end{align}

\end{corollary}

\begin{proof} It is enough to  compute $K_1$ and $K_2$ in all the cases and replace such values in  \eqref{simons5}. 
In case (1), $K_1=0,$ $K_2=-1,$ in case (2), $K_1=1,$ $K_2=-1,$ in case (3), $K_1=1,$ $K_2=0.$
\end{proof}


\begin{remark} \nero M. Batista \cite{Ba} proved a formula analogous to the result of Corollary \ref{product}  for surfaces in ${\mathbb H}^2\times{\mathbb R}$ 
and ${\mathbb S}^2\times{\mathbb R}.$ 

\nero In the case of  ${\mathbb H}^3\times{\mathbb R}$ and  ${\mathbb S}^3\times{\mathbb R},$ D. Fectu and H. Rosenberg  \cite{FR} proved a formula analogous to  the result of Corollary \ref{product} for surfaces with parallel mean curvature. 
\end{remark}

\section{A generalization of a R. Schoen, L. Simon, S.T. Yau's  inequality  for  finite index  hypersurfaces with constant mean curvature}
\label{SSY-section}

In this 
section we prove a generalization of one of the integral inequalities in \cite{SSY}, for finite index, constant mean curvature hypersurfaces in a  Riemannian manifold (Theorem \ref{general-ineq-theo}).  The analogous inequality  for minimal hypersurfaces is not explicitly stated in \cite{SSY}. There, it is a  key step towards the $L^p$ estimate of the norm of the second fundamental form of a minimal stable hypersurface.

We recall that we maintain the notation  and the conditions on $\mathcal N,$  established at the beginning of Section \ref{simons-formula}. 
From now on, for any $\Omega\subset M,$ we denote by $\Omega_+:=\{ p\in \Omega\ |  \ \varphi(p)\not=0\}.$

\begin{theorem}
\label{general-ineq-theo}
Let $M$ be a complete, noncompact  hypersurface with constant mean curvature $H$ and finite index, of a manifold $\mathcal N.$ 
Then, there exists a compact subset $K$ of $M$ such that, for any $q>-\frac{n+2}{n}$ and  for any $f\in C_0^{\infty}((M\setminus K)_+)$  one has 

\begin{align}\label{i}
&\int_{(M\setminus K)_+} f^{2}\varphi^{2q+2}({\mathcal A}\varphi^{2}+{\mathcal B}H\,\varphi+{\mathcal C}H^2+{\mathcal E})\notag \\ 
&\le {\mathcal D} \int_{(M\setminus K)_+}\varphi^{2q+2}|\nabla f|^{2}+ 
{\mathcal F}\int_{(M\setminus K)_+} f^{2} \varphi^{2q+1}+ 
{\mathcal G} \int_{(M\setminus K)_+} f^{2} \varphi^{2q}
\end{align}

where

$${\mathcal D}=\left(\frac{q+1+\tilde{\varepsilon}}{\tilde{\varepsilon}}\right)\left(\frac{2}{n(1+\varepsilon)}+(3q+2)-\tilde{\varepsilon}\right),$$
$${\mathcal A}= \left(\frac{2}{n(1+\epsilon)}+(2q+1)-\tilde{\varepsilon}\right)-(q+1)(q+1+\tilde{\varepsilon}),$$
$${\mathcal B}= -a_1(q+1)(q+1+\tilde{\varepsilon}),$$
$${\mathcal C}=n\left(\frac{2}{n(1+\varepsilon)}+(2q+1)-\tilde{\varepsilon}\right)+n(q+1)(q+1+\tilde{\varepsilon}),$$
$${\mathcal E}=a_2 (q+1)(q+1+\tilde{\varepsilon})+nK_{2}\,\left(\frac{2}{n(1+\epsilon)}+(2q+1)-\tilde{\varepsilon}\right),$$
$${\mathcal F}=2nK'(q+1)(q+1+\tilde{\varepsilon}), \ \  {\mathcal G}=-a_3(q+1)(q+1+\tilde{\varepsilon}),$$

with

$$a_1=\displaystyle{\frac{n(n-2)}{\sqrt{n(n-1)}}},\ \ \  a_2=\displaystyle{n\frac{(K_2-K_1)H}{2}+n(2K_2-K_1)},\ \ 
a_3=\displaystyle{\frac{n^2H(K_2-K_1)}{2}-\frac{n(n-1)}{2\varepsilon}(K_1-K_2)^2}$$

for any $\varepsilon,$ $\tilde\varepsilon>0.$  

Moreover if, in addition, $q\geq 0,$ then we can replace $(M\setminus K)_+$ with $M\setminus K$  and,  if $M$ is stable, then $K=\emptyset.$
\end{theorem}

 \begin{proof}
Using  the same notations as before, Simons' inequality (\ref{simons5}) yields 

\begin{align}
\label{simons6}
\varphi\Delta\varphi \geq \frac{2}{n(1+\varepsilon)}|\nabla\varphi|^2-\varphi^4-a_1H\,\varphi^3
+\varphi^2(nH^2+a_2)-2nK'\varphi+a_3.
\end{align}

 By Proposition \ref{stability-index-compact},  there exists a compact  subset $K$  in $M$ such that $M\setminus K$ is stable.  Notice that, if $M$ is stable, then $K=\emptyset.$ 
  
  Multiplying inequality \eqref{simons6} by $\varphi^{2q}\,f^{2},$  with $f\in C_0^{\infty}((M\setminus K)_+)$ and integrating  we obtain

\begin{align}
\label{a}
&-(2q+1)\int_{(M\setminus K)_+} \varphi^{2q}f^{2}|\nabla\varphi|^{2} -2\int_{(M\setminus K)_+}\varphi^{2q+1}f\langle\nabla f,\nabla\varphi\rangle \notag\\
&+\int_{(M\setminus K)_+} \varphi^{2q+4}f^{2}+a_1H\int_{(M\setminus K)_+}\varphi^{2q+3}f^{2}-(n\,H^{2}+a_2)\int_{(M\setminus K)_+} \varphi^{2q+2}f^{2}\\
&\notag+2\,K' \int_{(M\setminus K)_+}\varphi^{2q+1}f^{2}-a_3\,\int_{(M\setminus K)_+} \varphi^{2q}f^{2} \,\ge\, \frac{2}{ n(1+\varepsilon)}\int_{(M\setminus K)_+} |\nabla\varphi|^{2}\varphi^{2q}f^{2}.
\end{align}

We observe that, as we allow $q$ to be negative, we restrict to the subset $(M\setminus K)_+.$ If $q\geq 0,$ $f$ can be taken in $C_0^{\infty}(M\setminus K)$ and the set of integration in all the integrals in the following of the proof can be taken as $M\setminus K.$

Young's inequality gives for $\tilde{\varepsilon}>0$,

\begin{align}
\label{b}
|2\,\varphi^{2q+1}f\langle\nabla f,\nabla\varphi\rangle| &\le 2\,(\varphi^{q}f|\nabla\varphi|)(\varphi^{q+1}|\nabla f|)
\le \tilde{\varepsilon} \varphi^{2q}f^{2}|\nabla\varphi|^{2}+\frac{1}{\tilde{\varepsilon}}\varphi^{2q+2}|\nabla f|^{2}.
\end{align}

Using the estimate  \eqref{b}  in (\ref{a}),  we obtain

\begin{align}
\label{c}
&(\frac{2}{  n(1+\varepsilon)}+(2q+1)-\tilde{\varepsilon})\int_{(M\setminus K)_+}\varphi^{2q}f^{2}|\nabla\varphi|^{2}
\le \frac{1}{\tilde{\varepsilon}}\int_{(M\setminus K)_+} |\nabla f|^{2}\varphi^{2q+2}\\ 
&\notag+\int_{(M\setminus K)_+} (\varphi^{2}+a_1\,H\varphi-(n\,H^{2}+a_2))\varphi^{2q+2}f^{2}+2nK'\int_{(M\setminus K)_+}\varphi^{2q+1}f^{2}-a_3\int_{(M\setminus K)_+} \varphi^{2q}f^{2}.
\end{align}

The stability  inequality restricted to $(M\setminus K)_+$ yields, for any $\psi \in \mathcal{C}_{0}^{\infty}((M\setminus K)_+)$ 
\begin{equation}\label{d}
\int_{(M\setminus K)_+} |\nabla \psi|^{2}\,\ge\,\int_{(M\setminus K)_+}(|A|^{2}+{\rm Ric(\nu,\nu))}\psi^{2}\ge\int_{(M\setminus K)_+} (\varphi^{2}+n\,H^{2}+n\,K_{2})\psi^{2}.
\end{equation}

Taking $\psi=f\,\varphi^{q+1}$ in (\ref{d}), we get

\begin{align}\label{e}
\int_{(M\setminus K)_+}\varphi^{2q+2}|\nabla f|^{2}+&(q+1)^{2}\int_{(M\setminus K)_+} \varphi^{2q}f^{2}|\nabla \varphi|^{2}+2(q+1)\int_{(M\setminus K)_+} \varphi^{2q+1}f\langle\nabla f,\nabla\varphi\rangle\nonumber\\
&\ge \int_{(M\setminus K)_+}\varphi^{2q+4}f^{2}+n\,(H^{2}+K_{2})\int_{(M\setminus K)_+} \varphi^{2q+2}f^{2}.
\end{align}

Integrating Young's inequality \eqref{b} gives

\begin{align}\label{f}
2|\int_{(M\setminus K)_+}\varphi^{2q+1}f\langle\nabla f,\nabla \varphi\rangle|  \le \tilde{\varepsilon} \int_{(M\setminus K)_+} \varphi^{2q}f^{2}|\nabla \varphi|^{2}+\frac{1}{\tilde{\varepsilon}}\int_{(M\setminus K)_+} \varphi^{2q+2}|\nabla f|^{2}
\end{align}

 and using (\ref{f}) in (\ref{e}), we obtain

\begin{align}\label{h}
&-(q+1)(q+1+\tilde{\varepsilon})\int_{(M\setminus K)_+}  \varphi^{2q}f^{2}|\nabla \varphi|^{2} \notag\\
&\le(1+\frac{(q+1)}{\tilde{\varepsilon}})\int_{(M\setminus K)_+}  |\nabla f|^{2}\varphi^{2q+2}-\int_{(M\setminus K)_+}  f^{2}\varphi^{2q+2}(\varphi^{2}+n\,(H^{2}+K_{2})).
\end{align}

Now we make a linear combination of  the equations (\ref{c}) and (\ref{h})  in order to eliminate the term $\int \varphi^{2q}f^{2}|\nabla \varphi|^{2}.$ One needs  $(q+1)(q+1+\tilde{\varepsilon})>0$  and $(\frac{2}{n(1+\epsilon)}+(2q+1)-\tilde{\varepsilon})>0$  for $\tilde\varepsilon$ small enough. These conditions  are satisfied if

\begin{equation}
\label{condition-q}
q>-\frac{n+2}{2n}.
\end{equation}

Therefore,  $(q+1)(q+1+\tilde{\varepsilon})(\ref{c})+(\frac{2}{n(1+\epsilon)}+(2q+1)-\tilde{\varepsilon}) (\ref{h})$ gives

\begin{align}
&0\le [\frac{(q+1)(q+1+\tilde\varepsilon)}{\tilde{\varepsilon}}+(\frac{2}{n(1+\epsilon)})(1+\frac{q+1}{\tilde{\varepsilon}})]\int_{(M\setminus K)_+} |\nabla f|^{2}\varphi^{2q+2}\notag \\
&+ \int_{(M\setminus K)_+} [(q+1)(q+1+\tilde{\varepsilon})(\varphi^{2}+a_1H\,\varphi-(n\,H^{2}+a_2)\notag\\
&-(\frac{2}{n(1+\epsilon)}+(2q+1)-\tilde{\varepsilon})(\varphi^{2}+n\,H^{2}+n\,K_{2})] \varphi^{2q+2}f^{2}\notag \\
&+2nK'(q+1)(q+1+\tilde{\varepsilon})\int_{(M\setminus K)_+} \varphi^{2q+1}f^{2}-a_3(q+1)(q+1+\tilde{\varepsilon})\int_{(M\setminus K)_+} \varphi^{2q}f^{2}
\end{align}

which gives \eqref{i},  where  the constants  are as in the statement of  Theorem \ref{general-ineq-theo}.
\end{proof}

\begin{remark} As we observed in the Introduction:

\nero any of our results  of this section can be easily adapted to the case of $\delta$-stable minimal hypersurfaces. 
More generally  one can give a definition of 
$\delta$-stable  constant mean curvature hypersurface and study the corresponding inequalities.

\nero  an inequality analogous to \eqref{i} can be obtained for a hypersurface with constant $H_r$-curvature, that is the $r$-th symmetric function of the principal curvatures (in \cite{ASZ} and \cite{ASZ1} one can find some related results).  
\end{remark}

\begin{remark}\label{A-positive-rem}
In the following, any interesting application of inequality \eqref{i} is obtained for $\mathcal A>0.$  So, we determine  the condition on $q$ in order to have $\mathcal A>0.$ As $\mathcal A$ is continuous with respect to $\varepsilon$ and $\tilde \varepsilon,$  the sign of $\mathcal A$ for 
$\varepsilon=\tilde \varepsilon=0$ is preserved for $\varepsilon$ and $\tilde \varepsilon$ small, so we study the sign of $\left(\frac{2}{n}+2q+1\right)-(q+1)^2.$
 By a straightforward computation one obtains that $\mathcal A>0$ if and only if 
\begin{equation}
\label{A-positive}
-\sqrt{\frac{2}{n}}<q< \sqrt{\frac{2}{n}}
\end{equation}
\end{remark}

With a technique analogous to that of the proof of Theorem \ref{general-ineq-theo}  we are able to prove a kind of reversed H\"{o}lder  inequality. 
L. F. Cheung and D. Zhou \cite{CZ} proved such inequality in  the case 
$q=0.$

 \begin{theorem} \label{reduction-exponent}
 Let $M$ be a complete noncompact hypersurface immersed with constant mean curvature $H$ in a manifold with constant  curvature $c.$  
 Assume $M$ has finite index. Then there exists  a geodesic ball $B_{R_0}$ in  $M$  such that, for any $q\in\left[0,\sqrt{\frac{2}{n}}\right)$

\begin{align}
\label{reduction-ineq}
\int_{M\setminus B_{R_0}} \varphi^{2q+4}\leq {\mathcal S} \int_{M\setminus B_{R_0}} \varphi^{2q+2}
\end{align}

for some positive constant ${\mathcal S}$. 
Moreover, if $M$ is stable, then we can choose $B_{R_0}=\emptyset.$

\end{theorem}

\begin{proof}  Let $K$ be the subset of $M$ such that  $M\setminus K$ is stable.
Inequality \eqref{c} in the  present case yields ($\varepsilon=0,$ $K'=a_3=0,$ $a_2=nc,$ $q\geq 0$) 

\begin{align}
\label{c-red-exp}
(\frac{2} { n}&+(2q+1)-\tilde{\varepsilon})\int_{M\setminus K}\varphi^{2q}f^{2}|\nabla\varphi|^{2}
\leq\frac{1} { \tilde{\varepsilon}}\int_{M\setminus K} |\nabla f|^{2}\varphi^{2q+2}\notag\\
&+\int_{M\setminus K}\varphi^{2q+4}f^2
+a_1H\int_{M\setminus K}\varphi^{2q+3}f^2-n(H^2+c)\int_{M\setminus K}\varphi^{2q+2}f^{2}.
\end{align}

Young's inequality implies, for any  positive $\delta,$ 

\begin{equation}
\label{d-red-exp}
\varphi^{2q+3}f^2\leq\frac{\delta}{2}\varphi^{2q+4}f^2+\frac{1}{2\delta}\varphi^{2q+2}f^2.
\end{equation}

Replacing \eqref{d-red-exp} in \eqref{c-red-exp} one has

\begin{align}
\label{e-red-exp}
(\frac{2} { n}&+(2q+1)-\tilde{\varepsilon})\int_{M\setminus K}\varphi^{2q}f^{2}|\nabla\varphi|^{2}
\le \frac{1 }{\tilde{\varepsilon}}\int_{M\setminus K} |\nabla f|^{2}\varphi^{2q+2}\notag\\
&+\left(1+\frac{a_1H\delta}{2}\right)\int_{M\setminus K}\varphi^{2q+4}f^2
+\left(\frac{a_1H}{2\delta}-n(H^2+c)\right)\int_{M\setminus K}\varphi^{2q+2}f^2.
\end{align}

Inequality \eqref{h} in the present case yields ($K_2=c$)

\begin{align}\label{f-red-exp}
&-(q+1)(q+1+\tilde{\varepsilon})\int_{M\setminus K}  \varphi^{2q}f^{2}|\nabla \varphi|^{2} \notag\\
&\le(1+\frac{(q+1)}{\tilde{\varepsilon}})\int_{M\setminus K}  |\nabla f|^{2}\varphi^{2q+2}-\int_{M\setminus K}  
f^{2}\varphi^{2q+2}(\varphi^{2}+n(H^{2}+c))
\end{align}

 By doing  $((q+1)(q+1+\tilde \varepsilon)\eqref{e-red-exp}+(\frac{2}{n}+2q+1-\tilde\varepsilon)\eqref{f-red-exp}$
and rearranging terms, one has

\begin{equation}
\label{g-red-exp}
{\mathcal P}\int_{M\setminus K}\varphi^{2q+4}f^2\leq{\mathcal L}\int_{M\setminus K}  \varphi^{2q+2}|\nabla f|^2+
{\mathcal Q}\int_{M\setminus K}\varphi^{2q+2}f^2
\end{equation}

where 

$${\mathcal P}=\left(\frac{2}{n}+2q+1-\tilde\varepsilon\right)-(q+1)(q+1-\tilde\varepsilon)\left(1+\frac{a_1H\delta}{2}\right)$$

$${\mathcal L}=\frac{(q+1)(q+1-\tilde\varepsilon)}{\tilde\varepsilon}+\left(\frac{2}{n}+2q+1-\tilde\varepsilon\right)
\left(1+\frac{q+1}{\tilde\varepsilon}\right)$$

$${\mathcal Q}=(q+1)(q+1-\tilde\varepsilon)\left(\frac{a_1H}{2\delta}-n(H^2+c)\right)-(\frac{2}{n}+2q+1-\tilde\varepsilon)n(H^2+c).$$

For $\delta<<1$  and $\tilde\varepsilon<<1,$  the positiveness of $\mathcal P$ is equivalent to the positiveness of  
$\frac{2}{n}+2q+1-(q+1)(q+1),$  that is $q\in\left[0,\sqrt{\frac{2}{n}}\right).$ 

Fix $R_0$ such that $K\subset B_{R_0},$ so $M\setminus B_{R_0}$  is stable.   Define   $f\in C^{\infty}_0(M\setminus B_{R_0})$  to be a radial function such that $f\equiv 0$ on $B_{R_0}$ and on  $M\setminus B_{R_0+2R+1},$ 
$f\equiv 1$ on $B_{R_0+R+1}\setminus B_{R_0+1}$ and $|\nabla f|\leq C,$ with $C$ a positive constant. 


Replacing  $f$ in \eqref{g-red-exp} yields

\begin{equation}
\label{l-red-exp}
\int_{(B_{R_0+2R+1}\setminus B_{R_0})}\varphi^{2q+4}\leq{\mathcal S}\int_{ M\setminus B_{R_0}}  \varphi^{2q+2}
\end{equation}

where $\mathcal S=\frac{{\mathcal L}+C^2{\mathcal Q}}{\mathcal P}.$ 
Then, by letting $R$ go to infinity  in \eqref{l-red-exp}   one has

\begin{equation}
\label{m-red-exp}
\int_{M\setminus B_{R_0}}\varphi^{2q+4}\leq{\mathcal S}\int_{M\setminus B_{R_0}}  \varphi^{2q+2}
\end{equation}

 \end{proof}

\begin{remark} Theorem   \ref{reduction-exponent} has interesting consequences about  the relations between the volume entropy   
of a hypersurface $M$  and $\int_M \varphi^p$ for suitable $p$ (see \cite{INS}).
\end{remark}

\section{Caccioppoli's Inequalities}
\label{Caccioppoli}

  In this section   we give three consequences of   inequality \eqref{i}. Such consequences are Caccioppoli's type inequalities.  The first one is a generalization of Theorem 1 in \cite{SSY} and  involves $\varphi$ and the curvature of the ambient space. The second one involves, in addition, $|\nabla\varphi|.$  In order to obtain the third one, we restrict ourselves to the case    of constant  curvature ambient spaces and we improve inequality \eqref{i},  by estimating   carefully  the  involved constants.

\begin{theorem} [{\bf Caccioppoli's inequality of type I}]
\label{caccio-theo}
Let $M$ be a complete noncompact hypersurface immersed with constant mean curvature $H$ in a manifold $\mathcal N.$  Assume $M$  has finite index. 
Then, there exist a compact subset $K$ of $M$   and   constants $\beta_1,$  $\beta_2,$  $\beta_3,$ such that for every $f\in C_0^{\infty}((M\setminus K)_+)$ and $q>-\frac{n+2}{2n}$

 \begin{equation}
\label{i3}
\beta_1\int_{(M\setminus K)_+} f^{2q+4}\varphi^{2q+4} \leq \beta_2 \int_{(M\setminus K)_+}  |\nabla f|^{2q+4}+\beta_3\int_{(M\setminus K)_+}  f^{2q+4}.
\end{equation}

Moreover:

(i)   the constant $\beta_1$ is positive  if and only if $|q|<\sqrt{\frac{2}{n}},$ 

(ii)  if, in addition, $q\geq 0,$ then we can replace $(M\setminus K)_+$ with $M\setminus K.$

 \end{theorem}
 
 \begin{proof} Let $K$ be the compact set in $M$ such that $M\setminus K$ is stable. Let us first write \eqref{i}  taking only the  term with highest exponent of $\varphi$ in the left-hand side.

\begin{align}\label{i1}
&\notag \int_{(M\setminus K)_+}  {\mathcal A}f^{2}\varphi^{2q+4} \le {\mathcal D} \int_{(M\setminus K)_+} \varphi^{2q+2}|\nabla f|^{2}- {\mathcal B}H\int_{(M\setminus K)_+} f^2\varphi^{2q+3}\\
&-({\mathcal E}+{\mathcal C}H^2) \int_{(M\setminus K)_+}  f^{2}\varphi^{2q+2}+ {\mathcal F}\int_{(M\setminus K)_+}  f^{2} \varphi^{2q+1}+ {\mathcal G} \int_{(M\setminus K)_+} f^{2} \varphi^{2q}
\end{align}

We will transform all the  terms of the right-hand side of \eqref{i1}, using  Young's inequality, in order to obtain terms with $f^2\varphi^{2q+4},$ that can be reabsorbed by the left-hand side.  By Young's inequality one has

\begin{equation}
\label{young1}
|{\mathcal B}|H\varphi^{2q+3}\leq\varepsilon_1\varphi^{2q+4}+\frac{1}{\varepsilon_1}(|{\mathcal B}|H)^{2q+4} \ \ \ \ {\rm for\  any\   } \varepsilon_1>0,
\end{equation}

\begin{equation}
\label{young2}
-({\mathcal E}+{\mathcal C}H^2) \varphi^{2q+2}\leq \varepsilon_2\varphi^{2q+4}+\frac{1}{\varepsilon_2}|{\mathcal E}+{\mathcal C}H^2|^{q+2} \ \ \ \ {\rm for\  any\   } \varepsilon_2>0,
\end{equation}

\begin{equation}
\label{young3}
{\mathcal F}\varphi^{2q+1}\leq \varepsilon_3\varphi^{2q+4}+\frac{1}{\varepsilon_3}{\mathcal F}^{\frac{2q+4}{3}} \ \ \ \ { \rm for\  any\   } \varepsilon_3>0,
\end{equation}

\begin{equation}
\label{young4}
{\mathcal G} \varphi^{2q}\leq  \varepsilon_4\varphi^{2q+4}+\frac{1}{\varepsilon_4}|{\mathcal G}|^{\frac{q+2}{2}}\ \ \ \ {\rm for\  any\   } \varepsilon_4>0.
\end{equation}

Define $M'=\{ p\in M\  | \ f(p)\not=0\}.$ Then, on $M'$
\begin{align}
\label{young5}
&\notag\varphi^{2q+2}|\nabla f|^2=f^2\left[\varphi^{2q+2}\frac{|\nabla f|^2}{f^2}\right]
\\
&\leq \varepsilon_5f^2\varphi^{2q+4}+\frac{1}{\varepsilon_5}\frac{|\nabla f|^{2q+4}}{f^{2q+2}} \ \ \ \ {\rm for\  any\   } \varepsilon_5>0.
\end{align}

Now,  we replace \eqref{young1},  \eqref{young2},  \eqref{young3},  \eqref{young4},  \eqref{young5} in  \eqref{i1} and we get

\begin{align}\label{i2}
&({\mathcal A}-\sum_{i=1}^{4}\varepsilon_i-{\mathcal D}\varepsilon_5)\int_{(M\setminus K)_+}  f^{2}\varphi^{2q+4} \leq \frac{{\mathcal D}}{\varepsilon_5} \int_{(M\setminus K)_+\cap M'}\frac{|\nabla f|^{2q+4}}{f^{2q+2}}\\
&\notag+ \left(\frac{1}{\varepsilon_1}(|{\mathcal B}|H)^{2q+4}+\frac{1}{\varepsilon_2}|{\mathcal E}+{\mathcal C}H^2|^{q+2}+\frac{1}{\varepsilon_3}{\mathcal F}^{\frac{2q+4}{3}}+\frac{1}{\varepsilon_4}|{\mathcal G}|^{\frac{q+2}{2}}\right)
\int_{(M\setminus K)_+} f^2.\end{align}

One obtains inequality \eqref{i3}, after replacing  $f$ by $f^{q+2}$ in \eqref{i2} and letting 

$$\beta_1={\mathcal A}-\sum_{i=1}^{4}\varepsilon_i-{\mathcal D}\varepsilon_5,\    \  \beta_2=(q+2)^{2q+4}\frac{\mathcal D}{\varepsilon_5},$$
$$\beta_3=\frac{1}{\varepsilon_1}(|{\mathcal B}|H)^{2q+4}+\frac{1}{\varepsilon_2}|{\mathcal E}+{\mathcal C}H^2|^{q+2}+
\frac{1}{\varepsilon_3}{\mathcal F}^{\frac{2q+4}{3}}+\frac{1}{\varepsilon_4}|{\mathcal G}|^{\frac{q+2}{2}}.$$

Choosing   $\varepsilon_1,$ $\varepsilon_2,$ $\varepsilon_3,$    small enough and using Remark \ref{A-positive-rem}  one obtains (i). 
(ii) follows in the same way as in Theorem \ref{general-ineq-theo}.
\end{proof}

\begin{remark}
\label{growth-remark}
If $M$ is stable  and $q\geq 0,$ then inequality  \eqref{i3}  holds on $M$  for any $f\in C_0^{\infty}(M).$ Therefore, fixing  $t\in(0,1)$ and  
choosing a radial  function $f$  such that $f\equiv1$   on  the geodesic ball  $B_{ tR}$, $f\equiv 0$ on $M\setminus B_R$ and $f$ is linear on the annulus $B_R\setminus B_{ tR},$ 
one has 
\begin{equation}
\label{caccio-ball}
\beta_1\int_{B_{t R}} \varphi^{2q+4}\leq |B_R|\left(\frac{\beta_2}{(1-t)^{2q+4}R^{2q+4}}+\beta_3\right).
\end{equation}

Inequality \eqref{caccio-ball} yields interesting relations between $\int_M\varphi^{2q+4}$ and the volume entropy of $M$ (see \cite{INS}).

\end{remark}

Now we prove  a Caccioppoli's inequality involving the gradient of the norm of  $\varphi.$

 \begin{theorem}[{\bf Caccioppoli's  inequality of type II}]\label{caccio-nabla}
 Let $M$ be a complete noncompact hypersurface immersed with constant mean curvature $H$ in a manifold $\mathcal N.$  Assume $M$  has finite index. 
 Then, there exist  a compact subset $K$ of $M$ and positive constants $\beta_4,$ $\beta_5$  $\beta_6$ such that, for any function $f\in C_0((M\setminus K)_+)$ and $q>-\frac{n+2}{2n}$

\begin{align}
\label{l1}
{\mathcal A}\int_{(M\setminus K)_+}f^{2}|\nabla\varphi|^{2q}&\leq \beta_4\int_{(M\setminus K)_+} |\nabla f|^{2}\varphi^{2q+2}\notag\\
&+\beta_5\int_{(M\setminus K)_+} f^2\varphi^{2q+3}+\beta_6\int_{(M\setminus K)_+}  f^2.
\end{align}
Moreover:

(i)   the constant $\mathcal A$ is positive  if and only if $|q|<\sqrt{\frac{2}{n}},$ 

(ii)  if, in addition, $q\geq 0,$ then we can replace $(M\setminus K)_+$ with $M\setminus K.$

\end{theorem}

\begin{proof}

We sum the two inequalities \eqref{c} and \eqref{h} and we obtain

\begin{align}
\label{l}
&\left(\frac{2}{n(1+\varepsilon)}+(2q+1)-\tilde\varepsilon-(q+1)(q+1+\tilde\varepsilon)\right)\int_{(M\setminus K)_+} \varphi^{2q}f^{2}|\nabla\varphi|^{2}\notag\\
\leq &\left(1 +\frac{q+2}{\tilde\varepsilon}\right)\int_{(M\setminus K)_+}  |\nabla f|^{2}\varphi^{2q+2}-(2nH^2+a_2+nK_2)\int_{(M\setminus K)_+}  f^2\varphi^{2q+2}\notag\\
&+a_1H\int_{(M\setminus K)_+}  f^2\varphi^{2q+3}+2K'\int_{(M\setminus K)_+} \varphi^{2q+1}f^2-a_3\int_{(M\setminus K)_+} \varphi^{2q}f^2.
\end{align}

Using Young's inequality, as in the proof of  Theorem \ref{caccio-theo}, we reduce the terms containing $f^2\varphi^{2q},$ $f^2\varphi^{2q+1},$ $f^2\varphi^{2q+2},$   to a sum of terms containing $f^2\varphi^{2q+3},$ and $f^2.$ 

Then, there exist constants $\beta_5,$ $\beta_6,$ such that 

\begin{align*}
{\mathcal A}\int_{(M\setminus K)_+} \varphi^{2q}f^{2}|\nabla\varphi|^{2}&\leq \left(1 +\frac{q+2}{\tilde\varepsilon}\right)\int_{(M\setminus K)_+}  |\nabla f|^{2}\varphi^{2q+2}\notag\\
&+\beta_5\int_{(M\setminus K)_+}  f^2\varphi^{2q+3}+\beta_6\int_{(M\setminus K)_+}  f^2.
\end{align*}

Now we choose $\beta_4=(q+1)^{-2}\left(1 +\frac{q+2}{\tilde\varepsilon}\right)$ and we are done.

(i) and (ii) are obtained as in Theorem \ref{caccio-theo}.

\end{proof}

\begin{remark} 
 If $K_1=K_2=c$ and $H^2+c\geq 0,$ one obtains $\beta_6=0,$  so the integral in \eqref{l1} that does not contain $\varphi$ disappears. 
\end{remark}

When the ambient manifold $\mathcal N$ has  constant  curvature, by studying  carefully the sign of the  constants involved in \eqref{i}, one  obtains Caccioppoli's inequalities of type III. Let us start with the minimal case. 

In the following and without loss of generality, when the ambient space $\mathcal{N}$ has constant curvature $c,$ the constant $c$ will be  
-1, 0, +1,  according to its sign.

\begin{theorem} [{\bf Caccioppoli's inequality of type III - $H=0$}]
\label{caccio-min-theo}

Let $M$ be a complete noncompact minimal hypersurface immersed  in a manifold $\cal N$  with nonnegative constant  curvature $c.$  Assume 
$M$ has finite index. Then, there exists a compact subset $K$ of $M$  such that,  for any $f\in C^{\infty}_0(M\setminus K),$ one has

\begin{equation}
\label{jmin-bis}
{\mathcal A}\int_{M\setminus K}f^2 |A|^{2x+2}\leq {\mathcal D}\int_{M\setminus K}|A|^{2x}|\nabla f|^2
\end{equation}

provided $x\in\left[1,1+\sqrt{\frac{2}{n}}\ \right).$ 

Moreover if $x\in \left(1-\sqrt{\frac{2}{n}}, 1+\sqrt{\frac{2}{n}}\ \right)$ an  inequality analogous to \eqref{jmin-bis} holds with $(M\setminus K)_+$ instead of 
$M\setminus K.$ 
\end{theorem}

\begin{remark} In the case  where $M$ is stable, the analogous of inequality \eqref{jmin-bis} in ${\mathbb R}^{n+1}$ was proved by M. do Carmo and C. K. Peng  
\cite{DP1}.
\end{remark}

Before stating   next Theorem, we give two definitions:

\nero  For $\gamma=\frac{n-2}{n},$ $\mu=\frac{n^2}{4(n-1)}, $   let $g$ be the following function
\begin{equation}
\label{function-g}
g_n(x)=\frac{(2x-\gamma)^2-x^4}{(2x-\gamma)^2-\mu x^4}
\end{equation}

\nero Let $x_1$ and $x_2$ be the following real numbers

\begin{equation}
\label{roots}
x_1=\frac{2\sqrt{n-1}}{n} \left(1-\sqrt{1-\frac{n-2}{2\sqrt{n-1}}}\right), \  x_2=\frac{2\sqrt{n-1}}{n} \left(1+\sqrt{1-\frac{n-2}{2\sqrt{n-1}}}\right)
\end{equation}

\begin{theorem} [{\bf Caccioppoli's inequality of type III - $H\not=0$}]
\label{caccio-H}
Let $M$ be a complete  noncompact hypersurface immersed with constant mean curvature $H\not=0,$ in a manifold $\cal N$ with  constant  curvature $c.$ Assume 
$M$ has finite index  and $n\leq 5.$   
Then there exist  a compact subset $K$ in $M$  and a constant $\gamma$ such that,  for any $f\in C^{\infty}_0(M\setminus K)$

\begin{equation}
\label{j-bis}
\gamma\int_{M\setminus K} f^2 \varphi^{2x}\leq {\mathcal D}\int_{M\setminus K} \varphi^{2x}|\nabla f|^2
\end{equation}

provided either 

(1)  $c\geq0,$  $x\in[1,x_2).$

or 

(2)   $c=-1,$  $\varepsilon>0,$  $x\in[1, x_2-\varepsilon],$ $H^2\geq g_n(x).$

Moreover, if $n\leq 6$  and $x\in (x_1,x_2)$ in (1) (respectively $x\in [x_1+\varepsilon,x_2-\varepsilon]$ in (2)),  
 an  inequality analogous to \eqref{j-bis} holds with $(M\setminus K)_+$ instead of 
$M\setminus K.$ 
\end{theorem}

The Caccioppoli's  inequality of type III for $H\not=0$ is strongly different from the corresponding inequality for minimal hypersurfaces. 
Indeed, the power of $\varphi$ in the left hand-side term is $2x$ while in the minimal case it is $2x+2.$




\begin{proof} [Proof of Theorems  \ref{caccio-min-theo}, \ref{caccio-H}.]

 It is worthwhile to write here inequality \eqref{i}   and the value of the constant involved.   One has $K_{1}=K_{2}=c,$  $K'=0,$   $a_1=\frac{n(n-2)}{\sqrt{n(n-1)}},$ $a_2=nc,$ $a_3=0.$  
Furthermore,  we can take $\varepsilon=0$ in the Kato's inequality. Therefore, we get

$${\mathcal F}=0,\;\; {\mathcal G}=0 $$

$${\mathcal A}= \big(\frac{2}{n}+(2q+1)-\tilde\varepsilon\big)-(q+1)(q+1+\tilde\varepsilon).$$

$${\mathcal B}= -a_1\,(q+1)(q+1+\tilde\varepsilon).$$

$${\mathcal E}=c\ {\mathcal C}=cn\left[\left(\frac{2}{n}+(2q+1)-\tilde\varepsilon\right)+\,(q+1)(q+1+\tilde\varepsilon)\right].$$

 Then, inequality \eqref{i} yields (notice that we are assuming  $q\geq 0$)

\begin{equation}\label{j}
\int_{M\setminus K} f^{2}\varphi^{2q+2}({\mathcal A}\varphi^{2}+{\mathcal B}H\varphi+{\mathcal C}(H^{2}+c)) \le {\mathcal D} \int_{M\setminus K}\varphi^{2q+2}|\nabla f|^{2}
\end{equation}

We must find conditions on ${\mathcal A},$ $\mathcal B,$ $\mathcal C,$ $c,$ $H$ such that  the coefficient $\beta_7={\mathcal A}\varphi^{2}+{\mathcal B}H\varphi+{\mathcal C}(H^{2}+c)$  in \eqref{j} is positive. 
It is enough to do the computation for $\tilde\varepsilon=0,$ because  all the quantities are continuous with respect to  $\tilde\varepsilon.$  
We notice that, in some cases, 
the coefficient $\beta_7$ is positive without any condition on $H,$ while in some other cases, we have to look for  positiveness of $\beta_7$ 
provided  $H$ satisfies some conditions.

In order to simplify notation, we  let $x=q+1.$ Condition  \eqref{A-positive}  for the positiveness of $\mathcal A$ in terms of $x$ is
\begin{equation}
\label{condition-x}
\alpha_1:=1-\sqrt{\frac{2}{n}}<x<\alpha_2:=1+\sqrt{\frac{2}{n}}
\end{equation}

Let us study the different cases ($c=0,-1,1,$ $H=0,$ $H\not=0$).
\medskip

 (1) $c=0.$

 \nero $H=0:$ in this case  $\beta_7={\mathcal A}\varphi^{2}.$ We only need ${\mathcal A}=\displaystyle{-x^{2}+2x-(\frac{n-2}{n})}>0,$ 
so $x$ must satisfy condition \eqref{condition-x}.

\nero $H\not=0:$  in this case  $\beta_7={\mathcal A}\varphi^{2}+{\mathcal B}H\varphi+{\mathcal C}H^{2}.$ The quantity 
$\beta_7$ is positive for  any  value of $\varphi$ if and only if ${\mathcal A}>0,$  and 
$\Delta_0=H^2({\mathcal B}^2-4{\mathcal AC})< 0.$  As 
 
\begin{equation*}
\Delta_0=nH^{2}\left(x^{4}\frac{n^{2}}{n-1}-4\left(2x-\frac{n-2}{n}\right)^{2}\right),
\end{equation*}

then, $\Delta_0<0$ if and only if

$$x^4\frac{n^2}{n-1}< 4\left(2x-\frac{n-2}{n}\right)^2.$$

Condition \eqref{condition-x} guarantees that $x>\frac{n-2}{2n},$ then the previous inequality is equivalent to 

\begin{equation}
\label{polynomial}
x^2\frac{n}{\sqrt{n-1}}-4x+\frac{2(n-2)}{n}< 0.
\end{equation}

The  discriminant  of the polynomial in \eqref{polynomial} is positive if and only if  $n\leq 6.$

Then inequality \eqref{polynomial} is satisfied for  $n\leq 6$ and $x\in(x_1, x_2)$ where  $x_1$ and $x_2$ are defined in \eqref{roots} and 
are the roots of the polynomial in \eqref{polynomial}.

We observe that the value 1 is  contained in $ (x_1,x_2),$ if and only if $n\leq 5.$ Moreover, $\alpha_1\leq x_1< x_2\leq\alpha_2 $ with equality when $n=2.$
Therefore, for $n=2,$ the range of   $x$ is the same for any $H\geq 0.$ 
The conditions on $x$ are summed up in Table 1.

\


(2)  $c=-1.$

\nero $H\geq 0:$ in this case  $\beta_7={\mathcal A}\varphi^2+{\mathcal B}H\varphi+{\mathcal C}(H^2-1).$  
In order to study the positiveness of $\beta_7$  we compute the discriminant $\Delta_{-1}=({\mathcal B}^2-4{\mathcal AC})H^2+4{\mathcal AC}.$ 
The only case when one does not have a condition on $\varphi$ is ${\mathcal A}>0,$ $\Delta_{-1}< 0.$ We observe that for $H=0$ the last two conditions are never satisfied together. 

Looking at  Table 1, in order to have  ${\mathcal A}>0,$ $\Delta_{-1}< 0,$  one needs that $x\in (x_1, x_2)$  and $H^2>g_n(x), $ where  $g_n$ is the function that we 
have defined in \eqref{function-g}.
As  the supremum of $g_n$ on $(x_1,x_2)$  is $+\infty,$   in order to have some result one needs to restrict the interval of $x$ to  
$[x_1+\varepsilon,x_2-\varepsilon],$ for some positive $\varepsilon.$

\

(3)  $c=1.$

\nero $H=0:$ in this case,   $\beta_7={\mathcal A}\varphi^2+{\mathcal C},$ is positive for any value of $\varphi$ if and only if $x\in[\alpha_1,\alpha_2]$ (notice that ${\mathcal C}>0$).  

\nero  $H\not=0:$ in this case $\beta_7={\mathcal A}\varphi^2-{\mathcal B}H\varphi+{\mathcal C}(H^2+1).$   The quantity $\beta_7$ is positive for any value of $\varphi$ if and only if 
${\mathcal A}>0$ and $\Delta_1=H^2({\mathcal B}^2-4{\mathcal AC})-4{\mathcal A}{\mathcal C}<0.$
The two conditions are  verified   for any value of $H$ if $x\in(x_1,x_2).$ While, if $x\not\in(x_1,x_2)$ one needs $H^2\leq \frac{4{\mathcal A}{\mathcal C}}{{\mathcal B}^2-4{\mathcal A}{\mathcal C}}=-g_n(x)$ (see Table 1).

\begin{center}
\begin{variations}
x &\frac{n-2}{2n}&    &\alpha_{1}& & x_{1}&    &\frac{n-2}{n}&   &x_{2}&    &\alpha_{2}&    & +\infty \\
\filet
g_n'(x) &\z& &~-~& &\bb&~-~&\z&~+~&  \bb&  &~~+~~&      \\
\filet
g_n(x) & \h{\mu} &\d &  &-\infty  &\bb&   \h\pI &\d g(\frac{n-2}{n}) \c & \h\pI &\bb & -\infty&& \c &&&\h {\mu}\\
\filet
{\mathcal A} & \l& -& \z& &+&  &+&  &+& &\z &-&  \\
\filet
{\mathcal B}^{2}-4{\mathcal AC} & \l  &~~+~~& & &\z&  &~~-~~&  &\z &&~~+~~& \\
\end{variations}
\end{center}
\ 

\centerline {Table  1}
\

The results of Theorems \ref{caccio-min-theo}, \ref{caccio-H},  are obtained just putting together the previous estimates.

\end{proof}


Now we refine Caccioppoli's inequalities of type III (Theorem  \ref{caccio-min-theo} and \ref{caccio-H}) in order to allow more general exponents 
of $\varphi.$ This will be useful in Section \ref{applications}.

\begin{theorem} 
\label{theo-cacciomin-with-s}
Let $M$ be a complete  noncompact minimal   hypersurface immersed in a manifold  $\cal{N}$ with nonnegative constant   curvature $c.$  Assume $M$ has finite index.  
Let $\mu\in[2,\alpha_2+1)$ and $\eta>0,$ such that $\eta\mu\geq 1.$ 
Then,  there exists  a compact subset $K$ of $M,$ and a  positive constant $\delta_1$  such that  for any $f\in C^{\infty}_0(M\setminus K),$ one has
\begin{equation}
\label{cacciomin-with-s}
\int_{M\setminus K}f^{2\mu\eta}|A|^{2\mu}\leq \delta_1\int_{M\setminus K} |A|^{2\mu(1-\eta)}|\nabla f|^{2\mu\eta}.
\end{equation}

Moreover, if $\mu\in(\alpha_1+1,\alpha_2+1),$ an inequality analogous to \eqref{cacciomin-with-s} holds with $(M\setminus K)_+$ instead of 
$M\setminus K.$ 
\end{theorem}
\begin{proof}

When $\eta\mu=1$ inequality \eqref{cacciomin-with-s} is the same as inequality \eqref{jmin-bis} with $\mu=x+1$ and $\delta_1=\frac{\mathcal D}{\mathcal A}.$  Then assume $\mu\eta>1.$  
 Let $\mu=x+1$ in inequality \eqref{jmin-bis}  and   apply Young's inequality to the second integrand  of inequality \eqref{jmin-bis} on $M'$ as follows ($y\geq 0$) 

\begin{align*}
|A|^{2(\mu-1)}|\nabla f|^2&=f^2\left[|A|^{2(\mu-1)}\frac{|\nabla f|^2}{f^2}\right]=f^2\left[|A|^{2(\mu-1-y)}\frac{|A|^{2y}|\nabla f|^2}{f^2}\right]\\
&\leq f^2\left[\varepsilon |A|^{2(\mu-1-y)t}+\frac{1}{\varepsilon}\frac{|A|^{2ys}|\nabla f|^{2s}}{f^{2s}}\right]\\
\end{align*}

with $t=\frac{\mu}{\mu-1-y},$ $s=\mu\eta,$ $\eta=\frac{1}{y+1}.$

 Then replacing the previous inequality in inequality \eqref{jmin-bis} one has
 
 \begin{equation*}
 ({\mathcal A}-\varepsilon {\mathcal D})\int_{(M\setminus K)\cap M'} f^{2}|A|^{2\mu}\leq \frac{{\mathcal D}}{\varepsilon}\int_{(M\setminus K)\cap M'} |A|^{2\mu(\eta-1)}\frac{|\nabla f|^{2\mu\eta}}{f^{2\mu\eta-2}}
 \end{equation*}
 
Replacing $f$ by $f^{\mu\eta},$ one has inequality \eqref{cacciomin-with-s}, with $\delta_1=\frac{\mathcal D}{\varepsilon({\mathcal A}-\varepsilon{\mathcal D})}.$

\end{proof}

\begin{theorem} 
\label{theo-caccioH-with-s}
Let $M$ be a  complete noncompact hypersurface immersed with constant mean curvature $H$ in a manifold $\cal{N}$ with constant  curvature $c.$
Assume $M$ has finite index  and $n\leq 5.$
Then, there exists a compact $K\subset M$ and a positive constant $\delta_2$ such that
for any  $s\geq 1 $ and  any $f\in C^{\infty}_0(M\setminus K),$  one has

\begin{equation}
\label{caccioH-with-s}
\int_{M\setminus K}f^{2s}\varphi^{2x}\leq \delta_2\int_{M\setminus K} \varphi^{2x}|\nabla f|^{2s}
\end{equation}

provided either 

(1)  $c\geq0,$  $x\in[1,x_2),$

or 

(2)   $c=-1,$  $\varepsilon>0,$  $x\in[1, x_2-\varepsilon],$ $H^2\geq g(x).$

Moreover if $n\leq 6$ and $x\in (x_1,x_2)$ in (1) (respectively $x\in [x_1+\varepsilon,x_2-\varepsilon]$ in (2)) 
 an  inequality analogous to \eqref{j-bis} holds with $(M\setminus K)_+$ instead of 
$M\setminus K.$ 

\end{theorem}

\begin{proof}  When $s=1,$ inequality \eqref{caccioH-with-s} is the same as \eqref{j-bis} with $\delta_2=\frac{\mathcal D}{\gamma}.$ 
Then, assume $s>1$ and  apply  Young's inequality to the second integrand of inequality \eqref{j-bis}  on $M'$ as follows ($y\geq 0$) 

\begin{align*}
\varphi^{2x}|\nabla f|^2&=f^2\left[\varphi^{2x}\frac{|\nabla f|^2}{f^2}\right]=f^2\left[\varphi^{2(x-y)}\frac{\varphi^{2y}|\nabla f|^2}{f^2}\right]\\
&\leq f^2\left[\varepsilon \varphi^{2(x-y)t}+\frac{1}{\varepsilon}\frac{\varphi^{2ys}|\nabla f|^{2s}}{f^{2s}}\right]\\
\end{align*}

with $t=\frac{x}{x-y},$ $s=\frac{x}{y}.$

 Then replacing the previous inequality in \eqref{j-bis} one has 
 
 \begin{equation*}
 (\gamma-\varepsilon {\mathcal D})\int_{(M\setminus K)\cap M'} f^{2}\varphi^{2x}\leq \frac{{\mathcal D}}{\varepsilon}\int _{(M\setminus K)\cap M'}\varphi^{2x}\frac{|\nabla f|^{2s}}{f^{2s-2}}
 \end{equation*}
 
Replacing $f$ by $f^s,$ one has the result with $\delta_2=\frac{\mathcal D}{\varepsilon(\gamma-\varepsilon{\mathcal D})}.$
\end{proof}

\section{Applications  of the   Caccioppoli's  inequalities in the stable case}
\label{applications}

In this section we assume that $M$ is stable and we discuss  some consequences of Caccioppoli's inequality of type III. 
As the literature on the subject is wide and broken up, we  will compare our results with the old ones that we are aware of.

Notice that, when $M$ is   stable, all the results of the previous Sections hold taking  the compact subset 
$K=\emptyset.$  We split the discussion about the consequences of Caccioppoli's inequality  into two parts.  
First we deal with minimal hypersurfaces in a manifold of nonnegative constant curvature. We give conditions  on the total curvature, which ensure  that the hypersurface is totally geodesic.  Then we deal with hypersurfaces with constant mean curvature $H\not=0,$ in  $\mathbb{R}^{n+1}$, $\mathbb{S}^{n+1}$ and $\mathbb{H}^{n+1}.$  We give nonexistence results, provided some conditions on the total curvature are satisfied.
It will be clear in the following that all our results hold  when $\int_M \varphi^{p}$  is finite, for  suitable $p$ (see Remark \ref{finite-total}). 
We restrict ourselves to the complete noncompact case since in  $\mathbb{R}^{n+1},$ $\mathbb{S}^{n+1}$ and $\mathbb{H}^{n+1},$ the only weakly stable compact hypersurfaces of constant mean curvature are geodesic spheres \cite{BDE}.  

We recall that the  classification of  stable  constant mean curvature surfaces in $\mathbb{R}^{3},$ $\mathbb{S}^{3}$ and $\mathbb{H}^{3}$  is completely known. 
Stable, complete, orientable, minimal surfaces in ${\mathbb R}^3$ are planes, as it was proved  independently by  
 M. do Carmo and C.K. Peng \cite{DP}, D. Fischer-Colbrie and R. Schoen \cite{FCS} and A. V. Pogorelov \cite{Po}.  Later, A. Ros \cite{R} proved that 
there are no nonorientable stable minimal surfaces in ${\mathbb  R}^3.$ Finally, F. Lopez and A. Ros \cite{LR} proved that
weakly stable,  complete, noncompact,  constant mean curvature surfaces in ${\mathbb R}^3$ are planes.  


Let us now deal with the spherical case. There is no stable complete
minimal surface in ${\mathbb S}^3,$ as it can be deduced by  using  Theorem 4 in \cite{LR}  and   Theorem 5.1.1 in  \cite{Si}. 
 Later, K. Frensel  \cite {Fr} proved that there is no weakly stable complete
noncompact surface of constant mean curvature in ${\mathbb S}^3.$


Finally, in ${\mathbb H}^3$ one has the following results. In \cite{S}, da Silveira  proved that, in ${\mathbb H}^3,$ there are no weakly stable complete, noncompact surfaces with constant mean curvature $H\geq1$ except horospheres, while there are many examples of  weakly stable, complete surfaces with constant mean curvature $H\in(0,1).$ 
Furthermore, G. de Oliveira and the third author  
\cite{DS} found many examples of stable minimal  surfaces in ${\mathbb H}^3.$ 

It is proved in \cite{ENR}, \cite{X} and  \cite{Q} 
that, for $n=3,4,$  in ${\mathbb R}^{n+1}$ (respectively   in ${\mathbb H}^{n+1}$)
there is no finite index, complete,  noncompact hypersurface with constant mean curvature $H\not=0$ (respectively $H$ large enough). 
The analogous  problem  in higher dimension is still open. We give a partial answer to it, assuming $n\leq 5$ and some growth condition on 
$\int_M\varphi^{2x},$ for suitable $x.$ We observe that the arguments we use are of different nature from those used in  \cite{ENR} and do not allow us to deduce their results.
We also observe that  very little is known about noncompact stable hypersurface with constant mean curvature in the sphere ${\mathbb S}^{n+1},$ $n>2$ (see for instance  \cite{AD2}, where a nonexistence result is obtained under the assumption of   polynomial volume growth). 
Nevertheless, we obtain some results in this case, as well.

In the following, $B_R$ denotes,  as before, the geodesic ball in $M$ of radius $R.$ 

We start by studying some  consequences of Caccioppoli's inequality  for   $H=0.$ The first  result  is a consequence of Theorem \ref{theo-cacciomin-with-s}.

\begin{corollary}
\label{docarmo-peng-generalized}
Let $M$ be a complete noncompact minimal stable hypersurface immersed in a manifold with nonnegative constant curvature.  Assume that, for $\mu\in[2,\alpha_2+1),$ $\eta>0,$ $\eta\mu\geq1$ 

\begin{equation}\label{limit-mu}
\lim_{R\longrightarrow\infty}\frac{\int_{B_{2R}\setminus B_R} |A|^{2\mu(1-\eta)}}{R^{2\eta\mu}}=0.
\end{equation}

Then $M$ is totally geodesic. 

\end{corollary}

\begin{remark}

\label{power-lim}
Before proving Corollary \ref{docarmo-peng-generalized} we observe that,   in the denominator of \eqref{limit-mu}, one can take any power of $R$ smaller than $\eta\mu.$ In fact, for any $s\leq \eta\mu,$ one has 
\begin{equation*}
\frac{\int_{B_{2R}\setminus B_R} |A|^{2\mu(1-\eta)}}{R^{{2\mu\eta}}}\leq 
\frac{\int_{B_{2R}\setminus B_R} |A|^{2\mu(1-\eta)}}{R^{2s}}
\end{equation*}
\end{remark}

\begin{proof}[Proof of Corollary \ref{docarmo-peng-generalized}.]  Let $f\in C_0(M)$ such that $f\equiv 1$ on $B_R,$ $f\equiv 0$ on $B_{2R}\setminus B_R$ and $|\nabla f|\leq \frac{1}{R}.$ Replacing such $f$ in inequality  \eqref{cacciomin-with-s} yields, for any $\mu\in[2,\alpha_2+1),$ $\eta>0,$ $\eta\mu\geq 1$

\begin{equation*}
\int_{B_R}|A|^{2\mu}\leq \frac {\delta_1}{R^{2\eta\mu}}\int_{B_{2R}\setminus B_R}|A|^{2\mu(1-\eta)}
\end{equation*}

By hypothesis the second term in the previous inequality tends to zero as $R$ tends to infinity. Hence $|A|\equiv 0$ on $M$ and $M$ is totally geodesic.

\end{proof}

\begin{remark} 
\label{docarmo-peng}
\nero Notice that taking $\eta\mu=1$ in Corollary \ref{docarmo-peng-generalized} yields  that if,  for $x\in[1,\alpha_2),$  one has

\begin{equation*}
\limsup_{R\longrightarrow\infty}\frac{\int_{B_{2R}\setminus B_R} |A|^{2x}}{R^2}=0, 
\end{equation*}

then $M$ is  totally geodesic. 

\nero Corollary  \ref{docarmo-peng-generalized} is a generalization of  the result   by M. do Carmo and C. K. Peng,  stated in Theorem 1.3 of \cite{DP1}, for $\mathcal{N}=\mathbb R^{n+1}$, that is:  if there exists  
$ t\in(0,2\alpha_2)$ such that

 \begin{equation*}
\limsup_{R\longrightarrow\infty}\frac{\int_{B_{2R}\setminus B_R} |A|^2}{R^{t}}=0
\end{equation*}
then $M$ is totally geodesic. 
In fact,  this follows by taking    $\mu(1-\eta)=1,$ $t=2\eta\mu=2(\mu-1)\in(2,2\alpha_2)$ in   Corollary  
\ref{docarmo-peng-generalized}.  Then we can extend the range of the power  of $R$ in the denominator to  $(0,2\alpha_2),$ as in Remark  
\ref{power-lim}.
  \end{remark}

Now we state a particular case of  Corollary \ref{docarmo-peng-generalized}, which is 
a generalization to higher dimension of Theorem 2  in   \cite{LWe} by H. Li and G. Wei. Our generalization is different
from the one conjectured by  H. Li and G. Wei  for dimesion $n>3.$

\begin{corollary}
\label{li-wei-generalized}
Let $M$ be a complete  noncompact stable minimal  hypersurface immersed in a manifold with nonnegative constant  curvature  and let $n\leq 7.$   
 If there exists $t\in(2\alpha_2-1),$ such that
   
\begin{equation*}
\lim_{R\longrightarrow\infty}\frac{\int_{B_{2R}\setminus B_R} |A|^{3}}{R^t}=0,
\end{equation*}
  
then $M$ is totally geodesic. 

\end{corollary}

\begin{proof} In Corollary  \ref{docarmo-peng-generalized}, we take $2\mu(1-\eta)=3$ and define $t=2\mu\eta=2\mu-3.$ Then, in order to apply  Corollary  
\ref{docarmo-peng-generalized} one has to assume 
$2\leq t<2\alpha_2-1.$  Notice that $2<2\alpha_2-1$ if and only if $n\leq7.$  Now, as in Remark \ref{power-lim}, we extend the result to any  value 
$ t\in(0,2\alpha_2-1).$

\end{proof}

Now we deal with the case of constant mean curvature $H\not=0.$

The following result 
answers to a do Carmo's question in a particular case (see pg. 133 in \cite{Do}).

\begin{corollary}
\label{coro-docarmo-peng-H}
There is no complete noncompact stable  hypersurface $M$ with constant mean curvature $H$ in ${\mathcal N}={\mathbb R}^{n+1},$ 
$\mathbb S^{n+1}$  or 
 ${\mathbb H}^{n+1},$ $n\leq 5,$  provided there exists $s\geq1 $ such that

\begin{equation}\label{integral-polynomial}
\limsup_{R\longrightarrow\infty}\frac{\int_{B_{2R}\setminus B_R} \varphi^{2x}}{R^{2s}}=0
\end{equation}
and  either 

(1) ${\mathcal N}={\mathbb R}^{n+1}$  or $\mathbb S^{n+1},$   $H\not=0,$   $x\in[1, x_2),$

or  

(2) ${\mathcal N}={\mathbb H}^{n+1},$ $\varepsilon>0,$  $x\in[1,x_2-\varepsilon],$  $H^2>g_n(x).$  \end{corollary}

Before proving Corollary \ref{coro-docarmo-peng-H}, it is worthwhile to notice the following. 
Reasoning as in Remark \ref{power-lim} we can take any power of $R$ between zero and $\infty,$ in the denominator of 
\eqref{integral-polynomial}. This means that, in the hypothesis of 
Corollary \ref{coro-docarmo-peng-H},
$\int_{B_{2R}\setminus B_R} \varphi^{2x}$ can not be polynomial in $R.$

\begin{proof}[Proof of Corollary \ref{coro-docarmo-peng-H}.]
We use the same method as in the proof of Corollary \ref{docarmo-peng-generalized}, starting  with \eqref{caccioH-with-s} instead of \eqref{cacciomin-with-s}.
 Then we obtain that $\varphi\equiv0$ on $M,$ which  means that $M$ is totally umbilic. In case (1), it follows that  $M$ is contained either in a sphere or in a hyperplane. When ${\mathcal N}={\mathbb R}^{n+1},$ as  $M$ is complete  noncompact, then $M$ is a hyperplane and  $H=0.$  When ${\mathcal N}={\mathbb S}^{n+1},$   $M$ is a complete subset of a sphere, hence it is compact.  
 In case (2), it follows that $M$ is contained either in a sphere, or in a horosphere, or in a equidistant sphere. The inequality $H^2>g_n(x)\geq 1$ yields that $M$ can be only contained in a sphere. 
 As $M$ is complete and noncompact, this is a contradiction.

\end{proof}
\begin{remark}
\label{big-remark}
\nero The proof of Corollary \ref{coro-docarmo-peng-H} yields a result in  more general ambient manifolds. In fact, under the same conditions, if ${\mathcal N}$ 
has constant  curvature  and  is not necessarily simply connected, then $M$ is totally umbilical. 

\nero  Taking $s=1$ in Corollary \ref{coro-docarmo-peng-H}, one has that 
there is no complete noncompact stable  hypersurface $M$ with constant mean curvature $H$ in ${\mathcal N}={\mathbb R}^{n+1},$  
$\mathbb S^{n+1}$ or 
 ${\mathcal N}={\mathbb H}^{n+1},$ $n\leq 5,$  provided 

\begin{equation}\label{integral-polynomial-two}
\lim_{R\longrightarrow\infty}\frac{\int_{B_{2R}\setminus B_R} \varphi^{2x}}{R^{2}}=0
\end{equation}
and provided either 
(1) ${\mathcal N}={\mathbb R}^{n+1}$ or  $\mathbb S^{n+1},$ $H\not=0,$   $x\in[1, x_2),$
or  
(2) ${\mathcal N}={\mathbb H}^{n+1},$   $\varepsilon>0,$  $x\in[1,x_2-\varepsilon],$  $H^2>g_n(x).$ 
Also in this case, if ${\mathcal N}$ 
has constant curvature but it is not simply connected, then $M$ is totally umbilical. 

\nero  Taking    $x=1$ in  (1) of Corollary  \ref{coro-docarmo-peng-H}, we improve  the result   by H. Alencar and M. do Carmo,  stated in Theorem 4 of \cite{AD1}.  Furthermore,  M. do Carmo and D. Zhou  \cite{DZ1} stated a result  weaker than  (1) of Corollary \ref{coro-docarmo-peng-H} and, in their proof, they use wrongly Young's inequality (see equation (3.7) there).

\end{remark}











\begin{remark}
\label{finite-total} 
As we said before, many of the results of this article  apply to hypersurfaces $M$ such that  $\int_M\varphi^p<\infty,$ for suitable $p.$  
As an example we use Theorem \ref{reduction-exponent} in order to prove a result of L. F. Cheung, D. Zhou \cite{CZ} in a more direct and general form than the one contained in \cite{CZ}.  In fact, one can easily prove that for $n=3,4,5,$ in a simply connected manifold of constant  curvature $c$, any complete, weakly stable hypersurface $M$ with constant mean curvature satisfying $H^{2}+c>0$ and $\int_M\varphi^2<\infty,$  is a geodesic sphere. Indeed, Theorem \ref{reduction-exponent} ($q=0$)  yields $\int_M\varphi^4<\infty$  and  by Cauchy-Schwarz inequality, one has  $\int_M\varphi^3<\infty$ (since by hypothesis  $\int_M\varphi^2<\infty$). Using again Theorem \ref{reduction-exponent} ($q=\frac{1}{2}$), one obtains  $\int_M\varphi^5<\infty$. Then we apply Theorem 6.2 of \cite {BS} to derive the compactness of $M$. To conclude, we observe that a compact weakly  stable hypersurface of constant mean curvature  in a simply connected  space form is a geodesic sphere (see for instance \cite{BDE}).
\end{remark}

\subsection*
{Acknowledgements and Funds}
The second author would like to thank the {\it  Laboratoire
de math\'ematiques et physique th\'eorique} of Tours, where this article was prepared.
We also wish to thank Pierre B\'erard for pointing out to us  reference \cite{De}.

The first author was partially supported by the ANR through FOG 
project (ANR-07-BLA-0251-01).  The second author was partially supported by MIUR: PRIN 2007 Sottovariet\' a, strutture speciali e geometria
asintotica.

\textsc{Said Ilias}

{\em Universit\'e F. Rabelais, D\'ep. de Math\'ematiques, Tours

ilias@univ-tours.fr}

\textsc{Barbara Nelli}

{\em Dipartimento di Matematica Pura e Applicata, 

Universit\'a dell'Aquila

nelli@univaq.it}

\textsc{Marc Soret}

{\em Universit\'e F. Rabelais, D\'ep. de Math\'ematiques, Tours

 marc.soret@lmpt.univ-tours.fr}


\begin{thebibliography}{99}

\bibitem{AD} 
\textsc{H. Alencar, M.P. do Carmo:} {\it Hypersurfaces with constant mean curvature in spheres,}
Proc. Amer. Math. Soc. vol 120 (4) (1994) 1223-1229.

\bibitem{AD1} 
\textsc{H. Alencar, M.P. do Carmo:} {\it Hypersurfaces with constant mean curvature in space forms,}
An. Ac. Br. Sc.  66 (1994) 256-257.

\bibitem{AD2} 
\textsc{H. Alencar, M.P. do Carmo:} {\it Hypersurfaces of constant mean curvature with finite index and  volume of polynomial  growth,}
Arch. Mat. 60 (1993) 489-494, {\it Erratum to the previous article} 65 (1995) 271-272.

\bibitem{ASZ} \textsc{H. Alencar, W. Santos, D. Zhou:}  {\it Stable hypersurfaces with constant scalar curvature,}  
Proc. Am. Math. Soc. 138, 9 (2010) 3301-3312.


\bibitem{ASZ1} \textsc{H. Alencar, W. Santos, D. Zhou:}  {\it Curvature integral estimates for complete hypersurfaces,} 
arXiv:0903.2035v1 [math.DG] to appear in Illinois Jour. of Math.

\bibitem{BB} \textsc{L. Barbosa, P. B\'erard:} {\it Eigenvalue and "twisted" eigenvalue problems, applications to cmc surfaces,} 
J. Math. Pure Appl. 79, 5 (2000) 427-450.

\bibitem{BDE} \textsc{L. Barbosa, M. do Carmo, J. Eschenburg:} {\it Stability of hypersurfaces  of constant mean curvature in Riemannian manifolds,}
Math. Zeit. 197 (1988) 123-138.

\bibitem{Ba} \textsc{M. Batista:} {\it Simon's type equation in ${\mathbb H}^2\times{\mathbb R}$ and ${\mathbb S}^2\times{\mathbb R}$ and Applications,} arXiv:0904.2508v1 [math.DG].

\bibitem{B1} 
\textsc{P. B\'erard:} {\it Remarques sur l'equation de J. Simons,}
Differential Geometry (A Symposium in honour of Manfredo do Carmo on his 60th birthday), Pitman Monographs
Surveys Pure Appl. Math. vol 52, Longman Sci. Tech., Harlow (1991) 47-57.

\bibitem{B} \textsc{P. B\'erard:} {\it Simons' equation revisited,} Anais Acad. Bras. Ciencas (1994) 397-403.






\bibitem{BDS} 
\textsc{ P. B\'erard, M.P. do Carmo and W. Santos:}
{\it Complete hypersurfaces with constant mean curvature and finite total curvature}, Annals Glob. Anal. Geom. 16 (1998) 273-290.



\bibitem{BE}
\textsc{P.  B\'erard, R. Sa Earp:} {\it Minimal hypersurfaces in ${\mathbb H}^n\times{\mathbb R},$  total curvature and index}, 
arXiv:0808.3838v3 [math.DG]

\bibitem{BS}
\textsc{P. B\'erard, W. Santos:}
{\it Curvature estimates and stability properties of CMC-
submanifolds in space forms}, Mat. Contemp., 17(1999), 77-97.

\bibitem{Bern} \textsc{Bernstein:} {\it Sur un th\'eor\`eme de g\'eom\'etrie et ses applications aux
\'equations aux d\'eriv\'ees partielles du type elliptique,} Comm. de la Soc. Math. de Kharkov
(2\'eme s\'er) 15 (1915-1917)  38-45.

\bibitem{BGG} 
\textsc{ E. Bombieri, E. de Giorgi and E. Giusti:}
{\it   Minimal cones and the Bernstein problem }, Inv. Math.
Vol.  82 (1968) 243-269.


\bibitem{X} \textsc{X. Cheng:} {\it On Constant Mean Curvature Hypersurfaces with Finite Index,}
Arch. Math. 86 (2006) 365-374.

\bibitem{CDS} \textsc{L.F. Cheung, M. do Carmo,  W. Santos:} {\it  On the compactness of CMC-hupersurfaces with finite total curvature,} 
 Arch Math 73, 3 (1999) 216-222.


\bibitem{CZ} 
\textsc{L. F. Cheung, D. Zhou:} { \it Stable constant mean curvature hypersurfaces in $\mathbb{R}^{n+1}$ and $\mathbb{H}^{n+1}$, } Bull. Braz. Math. Soc., New series, 36 (1) (2005)  99-114.


\bibitem{S} \textsc{A. da Silveira:} {\it Stability of Complete Noncompact Surfaces with Constant Mean Curvature,} Math. Ann. 277 (1987) 629-638.


\bibitem{DS} \textsc{G. de Oliveira, M. Soret:} {\it Complete minimal surfaces in hyperbolic space,} Math.  Annalen, 311 (1998)  397-419.

\bibitem{De} \textsc{B. Devyver}: {\it On the finiteness of the Morse index for Schr\"odinger operators,} Manuscripta Math.
on line first, DOI: 10.1007/s00229-011-0522-1.

\bibitem{Do} \textsc{M. do Carmo:} {\it  Hypersurfaces of constant mean curvature,} Lecture Notes in Mathematics, vol. 1410 (1989) 128-144.

\bibitem{DP}  \textsc{M. do Carmo,  C. K. Peng:}  {\it   Stable
complete minimal surfaces in $R\sp{3}$ are planes }  Bull.
Amer. Math. Soc. (N.S.) 1, 6  (1979)  903-906.

\bibitem{DP1} \textsc{M. do Carmo, C. K. Peng:} {\it  Stable
 complete minimal hypersurfaces,}
Proc. of the 1980 Symposium of Differential Geometry, Beijing Vol.  1,2,3
(1982)   1349-1358.



\bibitem{DZ1} \textsc{M. do Carmo, D. Zhou:} {\it Bernstein-type theorems in hypersurfaces with constant mean curvature,} 
An. Acad. Bras. Cienc. 72 (3) (2000) 301-310, 
and {\it Erratum to: Bernstein-type theorems in hypersurfaces with constant mean curvature,} An. Acad. Bras. Cienc. 73 (3) 333-335. 

\bibitem{EH} \textsc{K. Ecker, G. Huisken:} {\it Interior curvature estimates for hypersurfaces of prescribed mean curvature,}
Ann. Inst. H. Poincar\'e Anal. Non Lin\'eaire 6 (4) (1989) 251-260.

\bibitem{ENR} \textsc{M. F. Elbert, B. Nelli, H. Rosenberg:} {\it Stable constant mean curvature hypersurfaces,}  
Proc. Amer. Math. Soc. 135,10 (2007)  3359�3366.

\bibitem{FR}\textsc{D. Fectu, H. Rosenberg:} {\it Surfaces with parallel mean curvature in ${\mathbb S}^3\times {\mathbb R}$ and ${\mathbb H}^3\times {\mathbb R}$,} 	arXiv:1103.6254v1 [math.DG]



\bibitem{FC} \textsc{D. Fischer-Colbrie:} 
{\it On complete minimal surfaces with finite index in three manifolds,} Invent. Math. 82 (1985) 121-132.

\bibitem{FCS} \textsc{D. Fischer-Colbrie, R. Schoen:} {\it The structure of complete stable minimal surfaces in 3-manifolds of
non-negative scalar curvature,}  Comm. on Pure and Appl. Math. 33 (1980)199�211.

\bibitem{Fr} \textsc{K. Frensel:} {\it Stable complete surfaces with constant mean  curvature}, Bull. of the Braz. Math. Soc. 27, 2 (1996) 129-144.


\bibitem{HS}\textsc{D. Hoffman, J. Spruck:} {\it Sobolev and isoperimetric inequalities for Riemannian submanifolds,} 
Comm. on Pure and Applied Math.  27,  6 (1974) 715-727.

\bibitem{INS} \textsc{S. Ilias, B. Nelli, M. Soret:} {\it On a question of Do Carmo  about  Stable  CMC Hypersurfaces,} in progress.



\bibitem{LWe}  \textsc{H. Li, G. Wei:} {\it Stable complete minimal hypersurfaces in $\R^4$,}  Matematica Contemporanea 28 (2005).

\bibitem{LR} \textsc{F. Lopez, A. Ros:} {\it Complete minimal surfaces with index one and stable constant mean curvature surfaces,}
Comment. Math. Helvetici 64 (1989) 34-43.

\bibitem{NS1} \textsc{ B. Nelli, M. Soret}: { \it The State of the Art of Bernstein's Problem},
Mat. Contemp. 29 (2005) 69�77.


\bibitem {NoS} \textsc{K. Nomizu, B. Smyth}: 
{\it A formula of Simons' type and hypersurfaces with constant mean curvature}, J. Diff. Geom. 3 (1969) 367-377.

\bibitem{O} \textsc{M. Okumura:} {\it Hypersurfaces and pinching problem 
on the second fundamental tensor,} Amer. Jour. of Math. 96 (1974) 207-213.


\bibitem{Po}\textsc{A.V. Pogorelov:} {\it On the stability of minimal surfaces,}  Soviet Math. Dokl. 24 (1981) 274�276.


\bibitem{Q}\textsc{D. Qintao:} {\it Complete hypersurfaces with constant mean curvature and finite index in hyperbolic 
spaces,} Acta Math. Scientia  31 B (1) (2011) 353-360.

\bibitem{R} \textsc{ A. Ros:} {\it One-sided comlpete stable minimal surfaces,}  Jour.  Diff. Geom.  74 (2006) 69-92.

\bibitem{RR} \textsc{A. Ros, M. Ritor\'e:} {\it Stable constant mean curvature  tori  and the isoperimetric problem in the three dimensional space form,} Comm. Math. Helvetici,  67, 1, (1992) 293-305.
 
\bibitem{SSY} \textsc{R. Schoen, L. Simon, S.T. Yau:}
{\it Curvature Estimates for Minimal Hypersurfaces},
Acta Math. 134 (1975) 275-288.

\bibitem{SZ} \textsc{Y.B.  Shen,  X.H.Zhu:} { \it On Stable Complete Minimal Hypersurfaces in $\R^{n+1}$},
American Jour. of Math. 120 (1998) 103-116.



\bibitem{Si} \textsc{J. Simons:} {\it Minimal varieties in Riemannian manifolds,} Ann. of Math.
88  (1968) 62-105.

\bibitem{Sp} \textsc{J. Spruck:} {\it Remarks on the stability of minimal hypersurfaces of $\R^n,$} Math. Zeit. 144  (1975) 169-174.





\end{thebibliography}
 \end{document}